\documentclass{amsart}

\usepackage{amsmath,amssymb}
\usepackage{amsthm}
\usepackage{latexsym}
\usepackage{times}
\usepackage{graphicx}
\usepackage{url}
\usepackage{setspace}
\usepackage{color}
\usepackage{placeins} 
\usepackage{algorithm,algorithmic}

\newcommand{\dd}{\textrm{d}}

\newcommand{\RR}{\mathbb{R}}

\newcommand\restr[2]{{
  \left.\kern-\nulldelimiterspace 
  #1 
  \vphantom{\big|} 
  \right|_{#2} 
  }}

\newtheorem{thm}{Theorem}[section]
\newtheorem{defn}[thm]{Definition}

\newtheorem{lem}[thm]{Lemma}
\newtheorem{pro}[thm]{Proposition}

\theoremstyle{remark}
\newtheorem*{rem}{Remark}

\author{Antonio Cicone}
\address{Dipartimento di Scienza e Alta Tecnologia, Universit\`a degli Studi dell'Insubria, via Valleggio 11, Como 22100, Italy}
\email{antonio.cicone@univaq.it}

\author{Hau-Tieng Wu}
\address{Department of Mathematics and Department of Statistical Science, Duke University, Durham, NC, USA. Mathematics Division, National Center for Theoretical Sciences, Taipei, Taiwan}
\email{hauwu@math.duke.edu}

\title{Convergence analysis of Adaptive Locally Iterative Filtering and SIFT method}

\begin{document}

\maketitle

\begin{abstract}

Adaptive Local Iterative Filtering (ALIF) is a currently proposed novel time-frequency analysis tool. It has been empirically shown that ALIF is able to separate components and overcome the mode-mixing problem. However, so far its convergence is still an open problem, particularly for highly nonstationary signals, due to the fact that the kernel associated with ALIF is non-translational invariant, non-convolutional and non-symmetric. Our first contribution in this work is providing a convergence analysis of ALIF. From the practical perspective, ALIF depends on a robust frequencies estimator, based on which the decomposition can be achieved.
Our second contribution is proposing a robust and adaptive decomposition method for noisy and nonstationary signals, which we coined the Synchrosqueezing Iterative Filtering Technique (SIFT). In SIFT, we apply the synchrosqueezing transform to estimate the instantaneous frequency, and then apply the ALIF to decompose a signal. We show numerically the ability of this new approach in handling highly nonstationary signals.
\newline\newline
\textbf{Keyboard}: Adaptive locally iterative filtering (ALIF), Synchrosqueezing Iterative Filtering Technique (SIFT), Synchrosqueezing transform, instantaneous frequency, nonstationary signal decomposition, time-frequency analysis
\end{abstract}

\section{Introduction}

In the past decades, the study of reliable time--frequency analysis methods for feature extraction from multi--component nonstationary signals has been a hot research topic. It is motivated by analyzing various challenging signals from various scientific fields.
For instance, a long-standing open problem in engineering regards the analysis and mitigation of multipath--corrupted measurements for the differential Global Navigation Satellite System (like GPS, Galileo, etc.) for vehicle navigation \cite{hofmann2007gnss}. In Medicine, one challenging problem is an automatic assessment of patients' general physical health from physiological time series like, for instance, the peripheral venous pressure \cite{wardhan2009peripheral} or the PPG signal \cite{cicone2017nonlinear}. In Physics, one important problem is identifying gravitational--wave signals of astrophysical origin in the gravitational-wave experiments in the Advanced LIGO--Virgo collaboration \cite{nakano2019comparison}. In Space Physics, there are several important open problems related to understanding the nature of plasma turbulence energy cascade, the solar wind origin, the interaction of the Sun with the inner heliosphere and, more in general, space weather \cite{papini2020multidimensional,camporeale2018coherent,materassi2019stepping,cicone2017Geophysics,spogli2019role}.
Among these far-from-exhaustive list of signals, one common challenge researchers face is that those signals are usually corrupted by large noise, and the signal is usually composed of various chirps, whistles, or multipaths, i.e. rapid changes in frequencies over time. To be more precise, in many practical applications, the signals under study contain several components with time-varying characteristics, for example, time-varying frequency, time-varying amplitude, and so son. Usually, these characteristics encode important system information under study, so researchers would like to capture these characteristics as accurately as possible.

The study of time--frequency (TF) analysis and signal decomposition is a long lasting line of research in signal processing. It has led over the decades to the development of many useful algorithms and approaches \cite{flandrin1998time,reviewWu}.
Most TF analysis algorithms belong to either linear or quadratic methods \cite{flandrin1998time}. In linear methods, the signal is studied via inner products with (or correlations with) a pre-assigned basis. The most well known methods are the short-time (or windowed) Fourier transform \cite{flandrin1998time} and the wavelet transform \cite{daubechies1992ten}. Among these methods, the chosen basis unavoidably leaves its peculiar footprint in the representation, which can badly affect the interpretation of the TF representation (TFR) when looking for properties of the signal. Moreover, the Heisenberg uncertainty principle limits the maximally achievable resolution of the TFR \cite{flandrin1998time}. While there is a trade-off between different choices of transforms or bases, it is challenging to find an ideal one \cite{flandrin1998time}.
In the quadratic approach, one may avoid introducing a basis and achieve a crisper and more focused TFR. However, as a side effect, reading the TFR of a multi-component signal becomes more complicated by the presence of interference terms \cite{flandrin1998time}. These interferences lead in most cases to even non-positivity in some parts of the TFR, which challenges this interpretation \cite{flandrin1998time}. This last flaw can be removed by some post-processing via kernel smoothing, which leads to the Cohen's class or Affine class \cite{cohen1995time,flandrin1998time}. However, this kernel smoothing approach again introduces unavoidably blurring in the resulting TFR. Furthermore, the extraction of a signal, or part of it, is much less straightforward for quadratic approaches when compared with linear approaches.

Motivated by these limitations, in the past decades, several nonlinear-type TF analyses have been proposed, for example, the reassignment method \cite{auger1995improving}, empirical mode decomposition (EMD) \cite{huang1998empirical}, the Blaschke decomposition \cite{Nahon:2000Thesis,coifman2017carrier}, the synchrosqueezing transform (SST) \cite{daubechies1996nonlinear,daubechies2011synchrosqueezed}, sparse time--frequency representation \cite{hou2011adaptive}, Iterative Filtering (IF) \cite{lin2009iterative}, Adaptive Locally Iterative Filtering (ALIF) \cite{cicone2016adaptive}, the empirical wavelet transform \cite{gilles2013empirical}, the variational mode decomposition \cite{dragomiretskiy2013variational}, the de-shape algorithm \cite{cicone2017nonlinear}, etc. We refer readers with interest to \cite{reviewWu} for a recent review in this direction. With these tools, we are able to produce a crisper and focused TFR to study a highly nonstationary multi-components signal and/or decompose components from a given signal. However, when the noise level is high, it is still a challenging problem, particularly when we come to signal decomposition problem.

In this article, we focus on the ALIF algorithm, a novel signal decomposition tool
inspired by the EMD \cite{huang1998empirical}. EMD is an iterative, local and adaptive data-driven method which has a ``divide et impera'' approach. The idea beyond EMD is simple, but powerful. We first iteratively divide the signal into several simple components via the ``sifting process'', and then each of them is analyzed separately in the TF domain via, for example, Hilbert transform \cite{huang1998empirical,huang2009instantaneous} or quadrature approach \cite{huang2009instantaneous}. While EMD seems to bypasses the Heisenberg-Gabor uncertainty principle \cite{flandrin1998time} and has been widely applied in practice, it contains several heuristics and ad hoc elements that make it hard to analyze mathematically.
Encouraged by the success of EMD, Iterative Filtering (IF) \cite{lin2009iterative} and its generalization, ALIF \cite{cicone2016adaptive}, were proposed as alternative iterative methods to capture the basic idea of EMD. IF and ALIF are mathematically tractable and have already been used effectively in a wide variety of applied fields, for instance, \cite{spogli2019role,materassi2019stepping,sfarra2019improving,cicone2017Geophysics,sharma2017automatic,li2018entropy,mitiche2018classification,yu2010modeling}.
The structure of the IF and ALIF algorithms resemble the EMD. The key difference is how the ``signal moving average'' is computed. In the IF and ALIF, it is the correlation of the signal itself with an a priori chosen filter function, instead of using the average between two envelopes. This apparently simple difference opens the doors to the mathematical analysis of IF and ALIF \cite{huang2009convergence,wang2013convergence,cicone2016adaptive,cicone2020study,cicone2020numerical,cicone2019spectral,cicone2020iterative,cicone2017multidimensional,cicone2019MFIF,cicone2020oneortwo,cicone2019nonstationary}.
The main challenge of IF is its limitation in extracting chirps from a signal. In fact, the algorithm was designed to extract, in a data driven fashion, simple components using only a narrow bandwidth of frequencies. For this reason, ALIF was developed as a flexible generalization of IF. ALIF overcomes this problem by allowing the frequency bandwidth to change locally and adaptively based on the signal \cite{cicone2016adaptive}.
This property makes the ALIF algorithm a promising and unique technique for decomposing multicomponent and highly nonstationary noisy signals.
However, there are at least two problems we need to address to let ALIF work properly. The first problem is how to theoretically guarantee its convergence \cite{cicone2019spectral}, like that for the IF method \cite{cicone2020numerical}. This challenge will be tackled in Section \ref{sec:results}.
Secondly, in its current formulation, it is not easy to directly identify which component should be extracted from a noisy signal, and we need a robust approach to estimate the frequency pattern before applying ALIF. This clearly limits considerably the ALIF usefulness in real life applications.
In this article, we propose to combine ALIF with SST, and equivalent methods, to extract the frequency pattern of the signal. This frequency pattern guides us in designing the adaptive filter in ALIF. By combining SST and ALIF, we produce an innovative algorithm, which we coined Synchrosqueezing Iterative Filtering Technique (SIFT), which will be elaborated in Section \ref{sec:SIFT}. We conclude this work with various numerical examples, Section \ref{sec:Examples}, followed by a conclusion section.

\section{Brief overview on the ALIF algorithm}

The ALIF method is an iterative procedure whose purpose is to decompose a signal $f$ into a finite number of ``simple components'' called
Intrinsic Mode Functions (IMFs). According to the qualitative description shown in \cite[Section~4]{huang1998empirical}, an IMF is a signal possessing the following two properties:
\begin{enumerate}
\item the number of extrema and the number of zero crossings must be either equal or differ at most by one; \item at any point, the mean value of the envelope connecting the local maxima and the envelope connecting the local minima must be zero.
\end{enumerate}

\begin{algorithm}
\caption{\textbf{(ALIF Algorithm)} $\textrm{IMFs = ALIF}(f)$}
\begin{algorithmic}
\STATE IMFs = $\left\{\right\}$
\STATE initialize the remaining signal $r=f$
\WHILE{the number of extrema of $r$ is $\geq 2$}
			\STATE for each $t\in\RR$ compute the kernel $\mathcal{K}_{\sigma}(t,\ x)$, whose support $\sigma$ depends on $t$ and is chosen based on $r$ itself
      \STATE $f_1=r$
      \STATE $m=1$
      \WHILE{the stopping criterion is not satisfied}
                  \STATE compute the moving average of $f_m$ as\\
                         $\mathcal{L}_{\sigma} f_m(t) := \int \mathcal{K}_{\sigma}(t,\ x) f_m(x) \dd x$
                  \STATE $f_{m+1}=f_m-\mathcal{L}_{\sigma} f_m(t)$
                  \STATE Increase $m$ by $1$
      \ENDWHILE
      \STATE Assign $\textrm{IMFs}\cup\{f_m\}$ to $\textrm{IMFs}$
      \STATE Assign $r-f_m$ to  $r$
\ENDWHILE
\STATE Assign $\textrm{IMFs}\cup\{r\}$ to $\textrm{IMFs}$
\end{algorithmic}
\label{alg:ALIF}
\end{algorithm}

Algorithm~\ref{alg:ALIF} shows the ALIF pseudocode. There, we introduce the kernel function $\mathcal{K}_{\sigma}(t,\ x)$ that represents a weighting function, where $\sigma$ is a continuous positive function depending on $t$. Here, $\sigma$ is chosen based on the signal under analysis such that $\inf_y \sigma(t)>0$. The kernel function allows us to compute the local moving average of the signal $f_m$ via
\[
\mathcal{L}_{\sigma} f_m:=\int \mathcal{K}_{\sigma}(t,\ x) f_m(x) \dd x.
\]
In the next section, we will present a typical kernel function and carry out the convergence analysis. The interested reader can find more details on the kernel function selection in \cite{cicone2016adaptive}.

The ALIF algorithm contains two loops. The inner loop captures a single IMF, and the outer loop produces all the IMFs constituting the given signal $f_0:=f$. Consider the first iteration of the ALIF outer loop. It extracts the first IMF from the signal $f_1=f$, the key idea is capturing the EMD by computing the moving average $\mathcal{L}_{\sigma} f_m$. Then it is subtracted from $f_m$ to capture its fluctuation part
$f_m-\mathcal{L}_\sigma f_m = f_{m+1}$. This procedure is repeated iteratively and, assuming convergence, the first IMF is obtained as $\textrm{IMF}_1=\lim_{m\to\infty} \left(I-\mathcal{L}_\sigma \right) f_m$. However, in practice, we cannot let $m$ go to $\infty$ and we have to use a stopping criterion, as pointed out in  Algorithm~\ref{alg:ALIF}. Assuming convergence, one can stop the inner loop at the first index $m$ such that the difference $f_{m+1}-f_m$ is small in some norm, or when a maximum number of iterations has been reached \cite{cicone2016adaptive}.

Once the first IMF is obtained we reapply the previous steps to the remaining signal $r=f-\textrm{IMF}_1$ to produce the second IMF. If we iterate this entire procedure we can obtain all the IMFs of $f$. The outer loop stops as soon as $r$ becomes a trend signal, i.e. it possesses at most one extremum. Clearly, the sum of all the IMFs of $f$ produced by the ALIF method, together with the final trend $r$, equals $f$.

We recall here that, if $\sigma(t)$, which is chosen at each iteration of the outer loop based on the remaining signal $r$, is selected to be a constant value $\sigma$ which does not depend on $t$, then the ALIF method boils down to the standard IF method, whose convergence and stability have been proved recently \cite{cicone2016adaptive,cicone2020numerical,cicone2020iterative}. Whereas it is still an open problem to prove the convergence of the ALIF method in its most general setting, i.e. when $\sigma$ does vary with $t$. We tackle this problem in the next section.

\section{ALIF convergence analysis}\label{sec:results}

We start this section with an alternative definition of ``simple function'' which is more rigorous than the qualitative IMF definition given in the previous section, and is at the foundation of the following analysis.

\begin{defn}[Intrinsic Mode Type (IMT) function \cite{daubechies2011synchrosqueezed}]
Fix $\varepsilon\geq0$. A function is called an $\varepsilon$-Intrinsic Mode Type (IMT) function if it satisfies
\[
f(t) = A (t)e^{i2\pi\phi(t)}\,,
\]
where $A$ and $\phi$ satisfy
\begin{itemize}
\item (Regularity Conditions) $A\in C^1(\RR)$, $\phi\in C^2(\RR)$;
\item (Boundedness Conditions) $A(t)>0$, $\phi'(t)>0$ for all $t\in\RR$;
\item (Growth Conditions)  $|A'(t)|\leq \varepsilon\phi'(t)$, $|\phi''(t)|\leq \varepsilon\phi'(t)$ for all $t\in\RR$, and $\sup_\RR \left|\phi''(t)\right|=M''<\infty$.
\end{itemize}
\end{defn}

We call $A(t)$ the {\em amplitude modulation} (AM), $\phi(t)$ the {\em phase function}, and $\phi'(t)$ the {\em instantaneous frequency} (iF) of the associated $\epsilon$-IMT function. Note that we use IF to denote the Iterative Filtering method and iF to denote the instantaneous frequency function. We emphasize that an IMT function is an IMF by the qualitative definition of IMF, while an IMF may not be an IMT function. Specifically, we can easily find a function $e^{i\phi(t)}$ so that $\phi''(t)> C\phi'(t)>0$ for a large $C>0$; that is, the ``frequency'' of a signal changes rapidly from time to time.  Also, unlike the IMF, the AM, phase and iF are all well-defined quantities without the identifiability issue. Indeed, if $f(t) = A (t)e^{i2\pi\phi(t)}=a(t)e^{i2\pi\varphi}$, then by taking absolute value we know that $A(t)=a(t)$ for all $t\in \mathbb{R}$. Then, after canceling $A(t)$, we know $\phi(t)=\varphi(t)$ up to a global phase $2K\pi$ for any $k\in \mathbb{K}$, and hence $\phi'(t)=\varphi'(t)$ for all $t\in \mathbb{R}$. We mention that in practice the signal is usually saved in the real form, that is, $f(t) = A (t)\cos(2\pi\phi(t))$. In this case, the identifiability issue is more complicated and has been resolved in \cite{Chen_Cheng_Wu:2014}. The relationship between $f(t) = A (t)e^{i2\pi\phi(t)}$ and $f(t) = A (t)\cos(2\pi\phi(t))$ is more complicated and delicate than it looks like. It is sometimes called the {\em Vakman's problem} \cite{Bedrosian:1962,Nuttall:1966,Picinbono:1997,Huang_Wang_Yang:2013}, and its discussion is out of the scope of this paper. We refer readers with interest to those cited papers. To simplify the discussion and focus on the ALIF algorithm, we stick to our IMT function setup.

\begin{defn}[Adaptive harmonic model]
Fix constants $ \varepsilon\geq 0$ and $d>0$. A function satisfies the $(\varepsilon,d)$-{\em adaptive harmonic model} (AHM) if it satisfies
\begin{equation}
f(t) = \sum_{\ell=1}^K f_\ell(t)\,,
\end{equation}
where $K\in \mathbb{N}$, $f_\ell(t)=A_\ell (t)e^{i2\pi\phi_\ell(t)}$ is an $\varepsilon$-IMT function, $\ell=1,\ \ldots,\ K$, and the following separation condition is satisfied:
\begin{itemize}\label{definition:adaptiveHarmonicMultiple}
\item (Separation Condition) $\phi_{\ell+1}'(t)-\phi'_\ell(t)\geq d$  for all $\ell=1,\ldots,K-1$, where $d>0$.
\end{itemize}
\end{defn}

We will study the convergence property of the ALIF on those functions satisfying the adaptive harmonic model.

\begin{defn}
Take a continuous positive function $\sigma:\mathbb{R}\to\mathbb{R}^+$ so that $\inf_t \sigma(t)>0$. Define
\begin{equation}\label{eq:L_operator}
    \mathcal{L}_{\sigma} f(y) := \int \mathcal{K}_{\sigma}(y,\ x) f(x) \dd x\,,
\end{equation}
where $\mathcal{K}_{\sigma}$ is the {\em Adaptive Local Iterative Filtering} kernel
\begin{eqnarray}
 \label{eq:K_ALIF} \mathcal{K}_{\sigma}(y,\ x) :=  \frac{1}{\sqrt{\pi}\sigma(y)}e^{-\frac{|y-x|^2}{\sigma(y)^2}}\,.
\end{eqnarray}
The {\em ALIF operator} associated with $\sigma$ and $K$ iterations is defined as
\[
\mathcal{S}^{(K)}_\sigma :=(I-\mathcal{L}_\sigma)^K\,,
\]
where $K\in \mathbb{N}$ and $I$ is the identity operator.
\end{defn}

We point out that, throughout the paper, to ease the notation and whenever there are no ambiguities, we drop the independent variable. For instance, in the previous formulas, $\sigma(y)$ is written as $\sigma$.

In ALIF, the kernel bandwidth depends on the location $y$. Due to the dependence on $y$, ALIF in general does not lead to the traditional bandpass filter.
Also, it is clear that the ALIF kernel \eqref{eq:K_ALIF} is in general asymmetric
\begin{equation}\label{eq:K_ALIF_nonsym}
    \mathcal{K}_{\sigma}({y},\ {x}) \neq \mathcal{K}_{\sigma}({x},\ {y})\,.
\end{equation}
We mention that this kind of non-translational invariant, non-convolutional, or non-symmetric kernel is also commonly encountered in other signal processing problems. For example, the kernel associated with the unsupervised dimension reduction algorithm {\em Locally Linear Embedding} (LLE) is also asymmetric, and is of the same kind of format \cite[Corollary 3.1]{Wu_Wu:2017}. Specifically, the kernel associated with LLE is determined by the data, and its bandwidth depends on the point. Moreover, the kernel morphology near the boundary is even different from that away from the boundary \cite{Wu_Wu:2019}. We refer the interested readers to \cite{Wu_Wu:2017,Wu_Wu:2019} for details.

\begin{rem}
If the kernel bandwidth is independent of $y$, that is, $\sigma(y)=c>0$ is constant, then the kernel is reduced to the {\em Iterative Filtering} kernel
\begin{eqnarray}
 \label{eq:K_IF} \mathcal{K}_{c}(y,\ x) = \frac{1}{\sqrt{\pi}c}e^{-\frac{|y-x|^2}{c^2}}\,,
\end{eqnarray}
with the bandwidth $c>0$ and the integral operator $\mathcal{L}_{\sigma}$ is reduced to the traditional convolutional kernel.
Unlike ALIF, the IF kernel \eqref{eq:K_IF} is always symmetric.
It is possible to consider a more general kernel setup; for example
\begin{eqnarray}
 \label{eq:K_gen} \mathcal{K}^{(\texttt{General})}_{\sigma}(y,\ x) :=  \frac{1}{\sqrt{\pi}\sigma(x,y)}e^{-\frac{|y-x|^2}{\sigma(x,y)^2}}\,.
\end{eqnarray}
In this setup, the kernel bandwidth depends not only on $y$, but also on $x$. In this paper, while the proof for the general kernel is a direct generalization of the proof in this paper, we focus on the ALIF kernel. An exploration of the general kernel setup, both in theoretical and numerical perspectives, will be reported in a future work.
\end{rem}

Intuitively, $\mathcal{S}^{(K)}_\sigma$ behaves like a high pass filter, while $\mathcal{L}_\sigma$ behaves like a low pass filter. However, it is clear that $\mathcal{S}_\sigma^{(K)}\neq I-\mathcal{L}_\sigma^k$, even if $\sigma$ is a constant function. To appreciate this difference, note that when $\sigma$ is a constant function, $\mathcal{S}^{(K)}_\sigma$ is reduced to a linear high pass filter and $\mathcal{L}_\sigma$ is reduced to a low pass filter. Under this assumption, by the cascade property of Gaussian kernel, we have
\[
I-\mathcal{L}_\sigma^K=I-\mathcal{L}_{K\sigma},
\]
while
\[
(I-\mathcal{L}_\sigma)^K=I+\sum_{l=1}^K(-1)^lc_l\mathcal{L}_{l\sigma}\,,
\]
where $c_l$ are the coefficients of the binomial expansion, which is clearly different from $I-\mathcal{L}_\sigma^K$.

We start with preparing some bounds.
\begin{lem}\label{Lemma Taylor expansion}
Fix $\epsilon\geq 0$. For an IMT function $f(t)=A(t)e^{i2\pi\phi(t)}$, we have
\begin{align}
\left|A(z)-A(y)\right|&\leq \varepsilon Q_y(z)\,\\
\left|\phi'(z)-\phi'(y)\right|&\leq \varepsilon Q_y(z)\,\nonumber
\end{align}
for any $y,z\in \mathbb{R}$, where
\[
Q_y(z):=\phi'(y) |z-y| + \frac{1}{2} M''(z-y)^2\,,
\]
and hence
\begin{align}\label{Freezing phase difference}
 \left|e^{i2\pi\phi(x)}-e^{i2\pi [\phi'(y)x-(\phi'(y)y-\phi(y))]}\right| \leq 2\pi \varepsilon \phi'(y)|y-x|^2\,.
\end{align}
Moreover, we have
\[
\left|e^{-|x-y|^2\phi'(z)^2}-e^{-|x-y|^2\phi'(y)^2} \right|\leq \varepsilon e^{-|x-y|^2\phi'(y)^2}|x-y|^2 \left[2\phi'(y) Q_y(z)+ \varepsilon Q_y^2(z)\right]\,.
\]

\end{lem}

\begin{proof}
The proof of the first part is based on Taylor's expansion, and can be found in \cite[Estimate 3.4]{}, so we skip it.
The second part is also based on Taylor's expansion.
\begin{align*}
  e^{-|x-y|^2\phi'(z)^2} &\, = e^{-|x-y|^2\left(\phi'(y)+\int_y^z \phi''(\xi)\dd \xi\right)^2}\\
  &\, = e^{-|x-y|^2\left[\phi'(y)^2+2\phi'(y)\int_y^z \phi''(\xi)\dd \xi+ \left(\int_y^z \phi''(\xi)\dd \xi\right)^2\right]}
\end{align*}
which implies that
\begin{align*}
 & e^{-|x-y|^2\phi'(z)^2}-e^{-|x-y|^2\phi'(y)^2}\\
  =&\, e^{-|x-y|^2\phi'(y)^2} \left(e^{-|x-y|^2\left[2\phi'(y)\int_y^z \phi''(\xi)\dd \xi+ \left(\int_y^z \phi''(\xi)\dd \xi\right)^2\right]}-1\right)\,.
\end{align*}
Hence, since $\left|e^{-x}-1\right|\leq x$ for all $x\geq 0$, it follows that
\begin{align*}
  &\left|e^{-|x-y|^2\phi'(z)^2}-e^{-|x-y|^2\phi'(y)^2} \right| \\
  \leq \,& e^{-|x-y|^2\phi'(y)^2} |x-y|^2\left|2\phi'(y)\int_y^z \phi''(\xi)\dd \xi+ \left(\int_y^z \phi''(\xi)\dd \xi\right)^2\right| \\
   \leq \,& e^{-|x-y|^2\phi'(y)^2} |x-y|^2\left[2\phi'(y)\left|\int_y^z \phi''(\xi)\dd \xi\right|+ \left|\int_y^z \phi''(\xi)\dd \xi\right|^2\right] \\
   \leq \,& e^{-|x-y|^2\phi'(y)^2}|x-y|^2 \left[2\phi'(y)\varepsilon Q_y(z)+ \varepsilon^2 Q_y^2(z)\right]\,,
\end{align*}
where the last bound comes from the first statement and the fact that $\phi'(z)-\phi'(y)=\int_y^z\phi''(\xi)d\xi$.
\end{proof}

Next, we consider the following ``freezing procedure''. Take an IMT function $f(x)=A(x)e^{i2\pi\phi(x)}$. We freeze $f$ at fixed $y$ by considering a harmonic function
\begin{equation}
f_y(x) := A_y e^{i2\pi [\phi'_yx-(\phi'_yy-\phi_y)]}\,,
\end{equation}
where $x\in \mathbb{R}$, $A_y:=A(y)$, $\phi_y:=\phi(y)$ and $\phi'_y := \phi'(y)$. Note that $f_y$ is a harmonic function with the amplitude $A_y$, frequency $\phi'_y$ and global phase $-\phi'_yy+\phi_y$ coming from the IMT function $f(x)$ at $y$. To simplify the notation and avoid confusion, we use the subscript ``${}_y$'' to emphasize that the associated quantity is a constant and does not vary from time to time.

\begin{pro}\label{prop:1}
Take an IMT function $f(t)=A (t)e^{i2\pi\phi(t)}$. Assume that $\varphi\in C^2(\mathbb{R})$ and also satisfies the phase condition of an IMT function. Then, we have
$$
(I-\mathcal{L}_{1/\varphi'})f(y) =\left[1- e^{-\pi^2\big(\frac{\phi'(y)}{\varphi'(y)}\big)^2}\right]f(y)+ E(y)\,,
$$
where $E(y)$ is defined in \eqref{Definition E function} and is bounded by $\varepsilon H(y)$ defined in \eqref{Defintion H function}.
\end{pro}

\begin{proof}
Fix $y\in \mathbb{R}$, and denote $\sigma_y:=1/\varphi'(y)$. We recall that, to ease the notation and whenever possible, we drop the independent variables. So, for instance, $\sigma(y)$ and $\varphi'(y)$ now read $\sigma$ and $\varphi'$. By a direct expansion, while in general $\mathcal{L}_{1/\varphi'}f(z) \neq  \mathcal{L}_{\sigma_y}f(z)$ when $z\neq y$, we know that
\begin{equation}\label{Proof eq Lvarphi=Lsigmay}
\mathcal{L}_{1/\varphi'}f(y) = \mathcal{L}_{\sigma_y}f(y) \,.
\end{equation}
By Lemma \ref{Lemma Taylor expansion}, the frozen $f$ at $y$, $f_y$, and $f$ are different by
\begin{align*}
|f_y(x)  -  f(x)| \leq\,& A_y \left| e^{i2\pi [\phi'_yx-(\phi'_yy-\phi_y)]} - e^{i2\pi \phi(x)}\right| + \left|[A_y  - A (x) ]e^{i2\pi \phi(x)}\right|\\
\leq\,& \varepsilon \left( 2\pi A_y  \phi'_y|y-x|^2 + \phi'(x)|y-x|+\frac{1}{2}M''|x-y|^2\right)
\end{align*}
for any $x\in \mathbb{R}$. Therefore, we have
\begin{align*}
& \left|\mathcal{L}_{\sigma_y}f(y) - \mathcal{L}_{{\sigma_y}}f_y(y)\right|\\
 \leq\,&  \int \mathcal{K}_{\sigma_y}(y,\ x) |f(x)  -  f_y(x)|  \dd x\\
\leq\,& \varepsilon \int \left[2\pi A (y) \phi'(y) |y-x|^2   + \phi'(x)|y-x| +\frac{1}{2}M'' |x-y|^2\right] \frac{1}{\sqrt{\pi}\sigma_y}e^{\frac{-|y-x|^2}{\sigma_y^2}} \dd x \\
=\,& \varepsilon\left[2\pi A (y) \phi'(y) J_2(y)+\phi'(y)J_1(y)+\frac{1}{2}M''J_2(y)\right] =\varepsilon H(y)\,,
\end{align*}
where
\begin{align}
J_1(y) :=\, & \int |y-x| \frac{1}{\sqrt{\pi}\sigma_y}e^{\frac{-|y-x|^2}{\sigma_y^2}} \dd x=\frac{1}{\sqrt{\pi}\sigma_y}\nonumber\\
J_2(y) :=\, & \int |y-x|^2 \frac{1}{\sqrt{\pi}\sigma_y}e^{\frac{-|y-x|^2}{\sigma_y^2}} \dd x=\frac{1}{2\sigma^2_y}\nonumber\\
H(y):=\,& \frac{\phi'(y)}{\sqrt{\pi}\varphi'(y)}+\frac{\pi A (y)\phi'(y)}{\varphi'(y)^2}+\frac{M''}{4\varphi'(y)^2} \,.\label{Defintion H function}
\end{align}
Clearly, $H\in C^1(\mathbb{R})$ by definition.
Finally,
\begin{align*}
\mathcal{L}_{\sigma_y}f_y(y) =\, & \int_\RR A_y e^{i2\pi [\phi'_yx-(\phi'_yy-\phi_y)]} \mathcal{K}_{\sigma_y}(y,\ x) \dd x\\
= \,& A_y e^{-i2\pi(\phi'_yy-\phi_y)} \int_\RR \frac{1}{\sqrt{\pi}\sigma_y} e^{i2\pi \phi'_yx} e^{\frac{-|y-x|^2}{\sigma_y^2}} \dd x
\end{align*}
By recalling that
$$
\mathcal{F}^{-1}\big( e^{-\alpha s^2}\big)(\xi) = \sqrt{\frac{\pi}{\alpha}} e^{-\frac{\pi^2}{\alpha}\xi^2}
$$
for any $\alpha>0$, we have
\begin{align*}
\mathcal{L}_{\sigma_y}f_y(y) =\, & A_y e^{-i2\pi(\phi'_yy-\phi_y)} \int \frac{1}{\sqrt{\pi}\sigma_y} e^{i2\pi \phi'_yx} e^{\frac{-|y-x|^2}{\sigma_y^2}} \dd x \\
=\,& A_y e^{i2\pi\phi_y} \int \frac{1}{\sqrt{\pi}\sigma_y}  e^{\frac{-s^2}{\sigma_y^2}} e^{i2\pi \phi'_ys} \dd s\\
=\,& f(y)  e^{-\pi^2 \phi'^2_y\sigma^2_y}\,.
\end{align*}
By recalling that $\sigma_y=1/\varphi'(y)$ and \eqref{Proof eq Lvarphi=Lsigmay}, we finish the proof and have that
$$
(I-\mathcal{L}_{1/\varphi'})f(y) = \left[1- e^{-\pi^2 \left(\frac{\phi'(y)}{\varphi'(y)}\right)^2} \right]f(y)  +  E(y)\,,
$$
where
\begin{equation}\label{Definition E function}
E(y):=\mathcal{L}_{1/\varphi'}f(y) - \mathcal{L}_{{\sigma_y}}f_y(y)= \int \mathcal{K}_{\sigma_y}(y,\ x) (f(x)  -  f_y(x))  \dd x
\end{equation}
is a smooth function bounded by $\varepsilon H(y)$.
\end{proof}

Proposition \ref{prop:1} describes how ALIF behaves before the iteration. Specifically, when applying $I-\mathcal{L}_{1/\varphi'}$ to an IMT function, it behaves like a high pass filter, with an error term. To quantify its behavior under iteration, we need to better control the error term. We have some observations. First, for an $\varepsilon$-IMT function $f(x)=A(x)e^{i2\pi\phi(x)}$, we claim that $\big(1- e^{-\pi^2 \phi'(y)^2/\varphi'(y)^2}  \big)f(y)$ is also an IMT function, with the parameter $c\varepsilon$. Indeed, since $\phi'$ and $\varphi'$ are both slowly varying by assumption, the function $1- e^{-\pi^2 \phi'(y)^2/\varphi'(y)^2}$ is also slowly varying. Therefore, we can view $\big(1- e^{-\pi^2 \phi'(y)^2/\varphi'(y)^2} \big)A(y)$ as the AM of $\big(1- e^{-\pi^2 \phi'(y)^2/\varphi'(y)^2} \big)f(y)$. By a direct expansion, we see that
\begin{align}
&\left|\Big[\big(1- e^{-\pi^2 \phi'(y)^2/\varphi'(y)^2} \big)A(y)\Big]'\right|\nonumber\\
=\,&\left|-e^{-\pi^2 \phi'(y)^2/\varphi'(y)^2}2\frac{\phi'(y)}{\varphi'(y)}\frac{\phi''(y)\varphi'(y)-\varphi''(y)\phi'(y)}{\varphi'(y)^2}A(y)+ \big(1- e^{-\pi^2 \phi'(y)^2/\varphi'(y)^2} \big)A'(y)\right|\nonumber\\
\leq\,&4 \varepsilon e^{-\pi^2 \phi'(y)^2/\varphi'(y)^2}\frac{\phi'(y)^2}{\varphi'(y)^2}A(y)+\varepsilon \big(1- e^{-\pi^2 \phi'(y)^2/\varphi'(y)^2} \big)\phi'(y)\nonumber\\
\leq\,&\varepsilon\left[\Big(4\frac{\phi'(y)A(y)}{\varphi'(y)^2}-1\Big)e^{-\pi^2 \phi'(y)^2/\varphi'(y)^2}\right]\phi'(y)\,.\nonumber
\end{align}
If $\frac{\phi'(y)A(y)}{\varphi'(y)^2}$ is uniformly bounded from above by $c>0$, then
$\big(1- e^{-\pi^2 \phi'(y)^2/\varphi'(y)^2} \big)f(y)$ is a $(c-1)\varepsilon$-IMT function.
Thus, when we apply $(I-\mathcal{L}_{1/\varphi'})$ to $(I-\mathcal{L}_{1/\varphi'})f(y)$, we may expect to apply Proposition \ref{prop:1}. {\em Formally}, we have
\begin{align*}
&(I-\mathcal{L}_{1/\varphi'})^2f(y)\\
 =\,&(I - \mathcal{L}_{1/\varphi'})\Big[1- e^{-\pi^2 \left(\frac{\phi'(y)}{\varphi'(y)}\right)^2} \Big]f(y) +  (I- \mathcal{L}_{1/\varphi'} )E(y)\\
 =\,&\Big[1- e^{-\pi^2 \left(\frac{\phi'(y)}{\varphi'(y)}\right)^2} \Big]^2f(y) + (c-1) E_1(y)+ (I- \mathcal{L}_{1/\varphi'} )E(y)\,,
\end{align*}
where $E_1(y)$ is the other error term. By this formal calculation, for $K>1$, we would expect to have that
\[
(I-\mathcal{L}_{1/\varphi'})^Kf(y)\approx \Big[1- e^{-\pi^2 \left(\frac{\phi'(y)}{\varphi'(y)}\right)^2} \Big]^K f(y)
\]
for some error term that is controllable. Below, we show this intuitive expectation.

\begin{pro}\label{prop:2}
Take the same setup and notation from Proposition \ref{prop:1}. Denote
\[
f_K(y):=\Big[1- e^{-\pi^2 \left(\frac{\phi'(y)}{\varphi'(y)}\right)^2} \Big]^Kf(y)
\]
to be the modulated IMT function $f(y)$.
Assume $\|\frac{\phi'}{\varphi'}\|_\infty<c'$ and $\|\frac{A}{\varphi'}\|_\infty<c''$ for $c'>0$ and $c''>0$. For $K\in \mathbb{N}$, we have
\begin{equation}
(I-\mathcal{L}_{1/\varphi'})f_K(y) =f_{K+1}(y)+ E_K(y)\,,\nonumber
\end{equation}
where $E_K(y)$ is bounded by $\varepsilon H_K(y)$ defined in \eqref{Defintion H_K function}.
\end{pro}
Clearly, when $K=0$, we recover Proposition \ref{prop:1} with $E_0(y)=E(y)$.

\begin{proof}
Fix $y\in \mathbb{R}$, and denote $\sigma_y:=1/\varphi'(y)$.
First, by a direct expansion, we have
\begin{align}
&\left|\partial_y\Big[\big(1- e^{-\pi^2 \phi'(y)^2/\varphi'(y)^2} \big)A(y)\Big]^K\right|\nonumber\\
\leq\,& \varepsilon\Big(1- e^{-\pi^2 \frac{\phi'(y)^2}{\varphi'(y)^2}} \Big)^{K-1}\left[4 \pi^2 K   e^{-\pi^2 \frac{\phi'(y)^2}{\varphi'(y)^2}}\frac{\phi'(y)A(y)}{\varphi'(y)^2}+ 1- e^{-\pi^2 \frac{\phi'(y)^2}{\varphi'(y)^2}} \right]\phi'(y)\nonumber\\
\leq\,&\epsilon \Big(1- e^{-\pi^2 c_1^2} \Big)^{K-1}\left[4 \pi^2 K   e^{-\pi^2 c_1^2} c_1c_2+ 1- e^{-\pi^2 c_1^2} \right]\phi'(y)\,,\nonumber
\end{align}
where the last inequality comes from the assumption. Denote
\[
c_K:=\Big(1- e^{-\pi^2 {c'}^2} \Big)^{K-1}\left[4 \pi^2 K   e^{-\pi^2 c'^2} c'c''+ 1- e^{-\pi^2 c'^2} \right]\,.
\]
Clearly, $\lim_{K\to \infty}c_K=0$.
Thus, if we view $A_K(y):=\big(1- e^{-\pi^2 \phi'(y)^2/\varphi'(y)^2} \big)^KA(y)$ as the AM of the $c_K\varepsilon$-IMT function $f_K(y)=A_K(y)e^{i2\pi\phi(y)}$, we can apply the same analysis for Proposition \ref{prop:1}.
By a similar bound, we have
\begin{align}
\left|A_K(z)-A_K(y)\right|&\leq c_K\varepsilon Q_y(z)\,.\nonumber
\end{align}
Fix $y$, and consider the frozen $f_K$ at $y$, denoted as $f_{K,y}$.
By a similar calculation, we have
\begin{align}
& \left|\mathcal{L}_{\sigma_y}f_K(y) - \mathcal{L}_{{\sigma_y}}f_{K,y}(y)\right| \nonumber\\
 \leq\,&  \int \mathcal{K}_{\sigma_y}(y,\ x) |f_K(x)  -  f_{K,y}(x)|  \dd x \nonumber\\
\leq\,& \varepsilon\left[\pi A_K (y) \phi'(y)\sigma_y^2+\frac{c_KM''\sigma_y^2}{4}+ \frac{c_K\phi'(y)\sigma_y}{\sqrt{\pi}}\right] =: \varepsilon H_K(y)\,.\label{Defintion H_K function}
\end{align}
Clearly, $H_K\in C^1(\mathbb{R})$ by definition. Similarly, we have
\begin{align*}
\mathcal{L}_{\sigma_y}f_{K,y}(y) = f_K(y)  e^{-\pi^2 \phi'^2_y/\varphi'^2_y}\,.
\end{align*}
Hence the claimed result that
$$
(I-\mathcal{L}_{1/\varphi'})f_K(y) = \left[1- e^{-\pi^2 \left(\frac{\phi'(y)}{\varphi'(y)}\right)^2} \right]f_K(y)  +  E_K(y)\,,
$$
where
\begin{equation}\label{Definition E_K function}
E_K(y):=\mathcal{L}_{1/\varphi'}f_K(y) - \mathcal{L}_{{\sigma_y}}f_{K,y}(y)= \int \mathcal{K}_{\sigma_y}(y,\ x) (f_K(x)  -  f_{K,y}(x))  \dd x
\end{equation}
is a smooth function bounded by $\varepsilon H_K(y)$.
\end{proof}

Next, we note the following recursive relationship.
\begin{align*}
(I-\mathcal{L}_{1/\varphi'})f(y)=\,&f_1(y)+ E_0(y)\\
(I-\mathcal{L}_{1/\varphi'})^2f(y)=\,&(I-\mathcal{L}_{1/\varphi'})f_1(y)+ (I-\mathcal{L}_{1/\varphi'}) E_0(y)\\
=\,&f_2(y)+ E_1(y)+ (I-\mathcal{L}_{1/\varphi'}) E_0(y)
\end{align*}
and in general when $K\geq 3$,
\begin{align*}
(I-\mathcal{L}_{1/\varphi'})^Kf(y)=f_K(y)+ \sum_{j=0}^{K-1}(I-\mathcal{L}_{1/\varphi'})^j E_{K-1-j}(y)
\end{align*}
Intuitively, if $\varphi'=c>0$ is constant, $(I-\mathcal{L}_{1/c})^{K-1} E(y)=\mathcal{F}^{-1}[(1-e^{-\pi^2\xi^2/c})^{K-1}\hat{E_0}(\xi)]$, which is a high pass filter.
We now bound $(I-\mathcal{L}_{1/\varphi'})^j E_{K-1-j}(y)$.
\begin{pro}\label{prop:3}
Take the same setup and notation from Proposition \ref{prop:2}. Denote $\tilde c:=1-e^{-\pi^2c'^2}$. We further assume that $\inf_y\varphi'(y)=c'''>0$. For $K\in \mathbb{N}$ and $j=0,\ldots,K-1$, we have
\begin{equation}
\|(I-\mathcal{L}_{1/\varphi'})^j E_{K-1-j}\|_\infty\leq 2^j C\tilde{c}^{K-1-j} \varepsilon \,.
\end{equation}
\end{pro}
\begin{proof}
By the definition of $A_K$, $H_K$ and assumptions, we have the following bounds. First,
\[
A_K(y)\leq \tilde c^KA(y)
\]
for all $y\in\mathbb{R}$. The second is a uniform bound of $H_K$:
\begin{align}
H_K(y)&\,=\pi A_K (y) \phi'(y)\sigma_y^2+\frac{c_KM''\sigma_y^2}{4}+ \frac{c_K\phi'(y)\sigma_y}{\sqrt{\pi}}\nonumber\\
&\,\leq \tilde{c}^K\left[\pi c'c''+\Big(\big(4 \pi^2 K c'c''-1\big)  e^{-\pi^2 c'^2} + 1 \Big)\Big(\frac{M''}{c'''^2}+\frac{c'}{\sqrt{\pi}}\Big)\right]\nonumber\\
&\,=: C\tilde{c}^K\,.\nonumber
\end{align}
Clearly, $\tilde{c}<1$, so $H_K(y)$ is exponentially smaller as $K$ increases. On the other hand, by a trivial bound, we have
\begin{equation}
|(I-\mathcal{L}_{1/\varphi'}) E_{K}(y)|\leq  \int \mathcal{K}_{1/\varphi'}(y,\ x)|E_K(y)-E_K(x)|\dd x \leq 2\epsilon C\tilde{c}^K\nonumber\,.
\end{equation}
Similarly,
\begin{align}
|(I-\mathcal{L}_{1/\varphi'})^2 E_{K}(y)|&\,\leq  \int \mathcal{K}_{1/\varphi'}(y,\ x)|(I-\mathcal{L}_{1/\varphi'})E_K(y)-(I-\mathcal{L}_{1/\varphi'})E_K(x)|\dd x \nonumber\\
&\,\leq \|(I-\mathcal{L}_{1/\varphi'})E_K\|_\infty+\|(I-\mathcal{L}_{1/\varphi'})E_K\|_\infty\nonumber\\
&\,\leq 4\epsilon C\tilde{c}^K\nonumber\,.
\end{align}
By iteration, we have the conclusion.
\end{proof}

Ideally, would expect a cancellation from $E_K(y)-E_K(x)$ by Taylor expansion when we bound $(I-\mathcal{L}_{1/\varphi'}) E_{K}(y)$. However, we do not have enough control on the Taylor expansion coefficients of $E_K$. As a result, the error grows exponentially when we iterate the filtering process.
Based on all the previous results we have our theorem proved.
\begin{thm}
Take a function $f(t) = \sum_{\ell=1}^L A_\ell (t)e^{i2\pi\phi_\ell(t)}$ that satisfies the AHM. Fix $1\leq k\leq L$. Suppose $\|\frac{\phi_\ell'}{\phi_k'}\|_\infty<c_\ell'$ and $\|\frac{A_\ell}{\phi_k'}\|_\infty<c_\ell''$ for $c_\ell'>0$ and $c_\ell''>0$.
For $m\in \mathbb{N}$, we have
\begin{align*}
\Big\|(I-\mathcal{L}_{1/\phi_k'})^m f-\sum_{\ell=1}^L \bigg[1- e^{-\pi^2 \left(\frac{\phi_\ell'(y)}{\phi_k'(y)}\right)^2} \bigg]^mA_\ell (t)e^{i2\pi\phi_\ell(t)} \Big\|_\infty\leq \varepsilon 2^{m-1}\sum_{\ell=1}^LC_\ell e^{\tilde c_\ell/2}\,,
\end{align*}
where $\tilde c_\ell:=1-e^{-\pi^2c_\ell'^2}$ and $C_\ell:=\pi c_\ell'c_\ell''+\big((4 \pi^2 K c_\ell'c_\ell''-1)  e^{-\pi^2 c_\ell'^2} + 1 \big)\big(\frac{M''}{c_\ell'''^2}+\frac{c_\ell'}{\sqrt{\pi}}\big)$.
\end{thm}

\begin{proof}
By linearity of the operator, we could focus on one IMT function and apply Proposition \ref{prop:3}.
For the $\ell$-th IMT component in $f$ and $K\in \mathbb{N}$, denote
\begin{equation}
E_{\ell,K}(y):= \int \mathcal{K}_{1/\phi'_k(y)}(y,\ x) (f_{\ell,K}(x)  -  f_{\ell,K,y}(x))  \dd x\,,\nonumber
\end{equation}
where
\[
f_{\ell,K}(y):=\Big[1- e^{-\pi^2 \left(\frac{\phi_\ell'(y)}{\phi_k'(y)}\right)^2} \Big]^Kf(y)\,.
\]
Thus, by Proposition \ref{prop:3}, we know that
\[
\|I-\mathcal{L}_{1/\phi_k'})^j E_{\ell,K-1-j}\|_\infty\leq 2^j C_\ell\tilde{c}_\ell^{K-1-j} \varepsilon\,.
\]
Thus,
\[
\Big\|\sum_{j=0}^{m-1}(I-\mathcal{L}_{1/\phi_k'})^j E_{\ell,m-1-j}\Big\|_\infty \leq \varepsilon C_\ell\sum_{j=0}^{K-1}2^j \tilde{c}_\ell^{m-1-j} \leq \varepsilon 2^{m-1}C_\ell e^{\tilde c_\ell/2}\,.
\]
\end{proof}

\section{SIFT algorithm}\label{sec:SIFT}

In this section we detail a new approach, called SIFT, which is based on a combination of SST and ALIF methods.
We start by summarizing the formula that allows estimate the iF of an IMT component from the STFT-based SST \cite{Wu:2011Thesis}. While it is possible to apply other variations of SST, like CWT-based SST \cite{daubechies2011synchrosqueezed} or de-shape SST \cite{cicone2017nonlinear}, or other suitable nonlinear-type TF analysis, we focus on the STFT-based SST to simplify the discussion. In short, the SST produces a meaningful and crisp TFR of a signal \cite{daubechies2011synchrosqueezed}, which allows us to robustly \cite{Chen_Cheng_Wu:2014} extract the iF of each oscillatory component.
Take $h$ to be the window function for the STFT, and denote $S^{(h)}_f\in \mathbb{C}^{N\times n}$ to be the SST of the given signal $f$ of length $N$, where the frequency axis is discretized into $n$ bins. Details of numerical implementation of SST can be found, for example, \cite{Chen_Cheng_Wu:2014}. We can extract the iF from the TFR of SST by the following formula:
\begin{equation}\label{CurveExtractionFormula}
c^*=\max_{c\in Z_{n}^{N}}\left(\sum_{m=1}^N \log\left[\frac{|S^{(h)}_f(c(m),m)|}{\sum_{i=1}^n\sum_{j=1}^N|S^{(h)}_f(j,i)|}\right] -\lambda \sum_{m=2}^N|c(m)-c(m-1)|^2\right)\,,
\end{equation}
where $Z_n=\{1,2,\ldots,n\}$ and $\lambda\geq 0$. Here, $c$ indicates a curve in the TFR $|S^{(h)}_f|\in \mathbb{R}_+^{N\times n}$, $\frac{|S^{(h)}_f(c(m),m)|}{\sum_{i=1}^n\sum_{j=1}^N|S^{(h)}_f(j,i)|}$ is a normalization step that normalizes the TFR to have a $L^2$ norm $1$, $\sum_{m=2}^N|c(m)-c(m-1)|^2$ quantifies the regularity of the extracted curve, and $\lambda$ is the penalty term. As is proved in \cite{Wu:2011Thesis,daubechies2011synchrosqueezed}, if the signal $f$ is a discretization of $\sum_{l=1}^LA_l(t)e^{i2\pi\phi_l(t)}$ that satisfies the AHM, where $A_l(t)e^{i2\pi\phi_l(t)}$ is the $l$-th $\epsilon$-IMT function, then the extracted curve $c^*$ is associated with the strongest IMT function. Based on the robustness property of SST \cite{Chen_Cheng_Wu:2014}, the extracted curve $c^*$ is a robust estimator of the iF of the strongest IMT function.
We mention that \eqref{CurveExtractionFormula} might not be the optimal approach for curve extraction, particularly for the purpose of estimating iF from the TFR, due to several limitations. For example, when there are multiple components, and we do not know the profile of each IMT function, there is no guarantee which iF will be extracted. Depending on the application, we may design a suitable curve extraction algorithm, probably by taking the background knowledge into account.

We are ready to describe the proposed SIFT algorithm. Suppose the input signal $f$ is a discretization of $\sum_{l=1}^LA_l(t)e^{i2\pi\phi_l(t)}$ that satisfies the AHM.
By SST, we extract the curve representing the iF of the highest frequency IMT, which is denoted as $\sigma_L$. Then, we apply ALIF to extract the first IMT from the signal under analysis. Based on the ALIF convergence proof that we provide in the previous section, we can guarantee the reconstruction of the $L$-th IMT function. We do so by using $\sigma_L$ in the moving average operator \eqref{eq:K_ALIF}. Then we subtract the IMT from the original signal to generate a remainder $r(t)$. We can now apply the previous steps to the remainder signal $r$ until there are no more oscillatory components in it.
The SIFT pseudocode is presented in Algorithm \ref{alg:SIFT}, where we use the notation introduced in Section \ref{sec:results}.

\begin{algorithm}
\caption{\textbf{(SIFT Algorithm)} $\textrm{IMTs = SIFT}(f)$}
\begin{algorithmic}
\STATE IMTs = $\left\{\right\}$
\STATE $l=1$
\STATE Initialize the remaining signal $r=f$;
\WHILE{$r$ contains oscillatory components}
      \STATE Compute $S^{(h)}_f(t,\eta)$, i.e. the TFR of $r$ via SST;
      \STATE Identify the highest frequency meaningful iF curve $\sigma_l(t)$ from the TFR;
	  \STATE For each $t\in\RR$ compute the kernel $\mathcal{K}_{\sigma_l}(t,\ x)$, where $\sigma_l$ depends on $t$;
      \STATE $f_1=r$
      \STATE $m=1$
      \WHILE{the stopping criterion is not satisfied}
                  \STATE Compute the moving average of $f_m$ as\\
                         $\mathcal{L}_{\sigma_l} f_m(t) := \int \mathcal{K}_{\sigma_l}(t,\ x) f_m(x) \dd x$
                  \STATE $f_{m+1}=f_m-\mathcal{L}_{\sigma_l} f_m(t)$
                  \STATE Increase $m$ by $1$
      \ENDWHILE
      \STATE Assign $\textrm{IMTs}\cup\{f_m\}$ to $\textrm{IMTs}$
      \STATE Assign $r-f_m$ to $r$
      \STATE Increase $l$ by $1$
\ENDWHILE
\STATE Assign $\textrm{IMTs}\cup\{r\}$ to $\textrm{IMTs}$
\end{algorithmic}
\label{alg:SIFT}
\end{algorithm}

With the decomposed IMT functions by SIFT, we can follow the idea in \cite{coifman2017carrier} and produce a TFR, which we coined SIFT-TFR, in the following way. Denote $\tilde{f}_k$ to be the decomposed $k$-th IMT function. The SIFT-TFR is
\[
\mathsf{S}^{(h)}_f(t,\eta)=\sum_{l=1}^L S^{(h)}_{\tilde{f}_l}(t,\eta);
\]
that is, we evaluate the TFR of each $\tilde{f}_l$ by SST, and then sum them together.
Note that $\mathsf{S}^{(h)}_f(t,\eta)$ is different from ${S}^{(h)}_f(t,\eta)$.

To have a fair comparison and demonstrate the benefit of SIFT,
we recall the reconstruction formula for the $k$-th IMT function via the STFT-based SST \cite[Theorem 2.3.14]{Wu:2011Thesis} (see \cite{daubechies2011synchrosqueezed} for the reconstruction formula for the CWT-based SST).
First, the reconstruction formula for the $k$-th IMT function in the continuous setup is
\begin{equation}\label{eq:SST_recon_formula}
    \frac{1}{g(0)}\!\!\!\!\!\!\!\!\!\!\!\!\!\!\!\!\!\int\limits_{\hskip 1cm \phi_k'(t)-\epsilon^{1/3}}^{\hskip 1cm \phi_k'(t)+\epsilon^{1/3}}\!\!\!\!\!\!\!\!\!S^{(h)}_f(t,\eta)d\eta=A(t)e^{i2\pi\phi_k(t)}+O(\epsilon)
\end{equation}
By the robustness property of SST \cite{Chen_Cheng_Wu:2014}, the reconstructed IMT function is robust to the existence of noise.
Suppose we discretize the function with the sampling period $\Delta_t>0$, and discretize the frequency axis of the SST with the discretization bin $\Delta_\xi>0$. Under this setup, the reconstruction formula \eqref{eq:SST_recon_formula} for the $k$-th IMT function at time $t=l\Delta_t$ is discretized into:
\begin{equation}\label{eq:SST_recon_formula2}
\frac{1}{g(0)}\!\!\!\!\!\!\!\!\!\!\!\!\!\!\!\!\!\int\limits_{\hskip 1.2cm \phi_k'(l\Delta_t)-\epsilon^{1/3}}^{\hskip 1.2cm \phi_k'(l\Delta_t)+\epsilon^{1/3}}
\!\!\!\!\!\!\!\!\!\!\!\!\!\!\!\!\!\!\!S^{(h)}_f(l\Delta_t,\eta)d\eta\approx \frac{\Delta_\xi}{g(0)} \sum_{q\in B}S^{(h)}_f(l\Delta_t,\,q\Delta_\xi)
\end{equation}
where
\[
B=\{q; \phi_k'(l\Delta_t)-\epsilon^{1/3}\leq q\Delta_\xi \leq \phi_k'(l\Delta_t)+\epsilon^{1/3}\}.
\]
In practice, since $\epsilon$ is usually unknown, $B$ could be chosen to be
\[
B=\{q; \phi_k'(l\Delta_t)-b\leq q\Delta_\xi \leq \phi_k'(l\Delta_t)+b\},
\]
where $b>0$ is the chosen frequency range. In practice, usually we take $b$ small enough so that $\phi'_k(l\Delta_t)-b$ and $\phi'_{k-1}(l\Delta_t)$ are sufficiently separated.

\section{Numerical Examples}\label{sec:Examples}

In the following we show the performance of SIFT method compared with SST and Bandpass Filter (BPF) algorithm by means of a few numerical examples.

In all the proposed examples we intentionally remove the first and last second results values. This is because all the compared techniques have their problems handling the boundary issue. Recently, a few papers concerning the analysis of the boundary effects in a signal decomposition have been published in the literature \cite{cicone2020study,stallone2020new}. More has to be done in this direction. However, since the analysis of these effects is out of the scope of this work, we do not take it into account and postpone their analysis to future work.

We now summarize the implementation details of the following examples. The sampling rate is fixed to 100 Hz, and the noise we introduce is always the additive standard white Gaussian noise. Regarding the SST algorithm parameters setting, we consider the window $h$ to be of length $377$, which comes from discretizing the Gaussian window $\exp(-t^2/2)$ supported on $[-6,6]$ seconds into equally spaced sample points, unless differently stated. The frequency axis is discretized every 0.01Hz. In the SST reconstruction, we choose $b=0.1$Hz in \eqref{eq:SST_recon_formula2}, which implies that in the code we set the SST bandwidth to 10. For ALIF, in order to extract IMT corresponding to the instantaneous frequency curve $\sigma_l(t)$, we set the xi parameter to $1.4$. This parameter is associated with the initial tuning of a given filtering window in ALIF, and does not depend on the signal under analysis. For further details the interested reader can refer to \cite{cicone2016adaptive}.
In the SST reconstruction we set the absolute frequency band to $0.3$ Hz.

Concerning the curve extraction problem, in the following examples, for simplicity, we always make use of the known instantaneous frequency ground truth as the profile for curve extraction; that is, we search for the optimal curve around the ground truth. We postpone to future works the identification of the optimal strategy in this context, like simultaneously extracting all curves or identify the curves we have interest.

All data and codes used in this work are available for download at \url{https://github.com/Acicone/SIFT_paper}

\subsection{Example 1}
In the first example, we apply the applied SIFT to a signal which contains two nonstationary components $s_1$ and $s_2$, where the iF's do not overlap, but the ranges of the iF's overlap. We assume the signal is clean without any noise contamination. The signal, its simple components and their instantaneous frequencies are depicted in the first two rows of Figure~\ref{fig:Ex2_wNoise_signal}.

\begin{figure}[ht]
\centering
\includegraphics[width=0.9\linewidth]{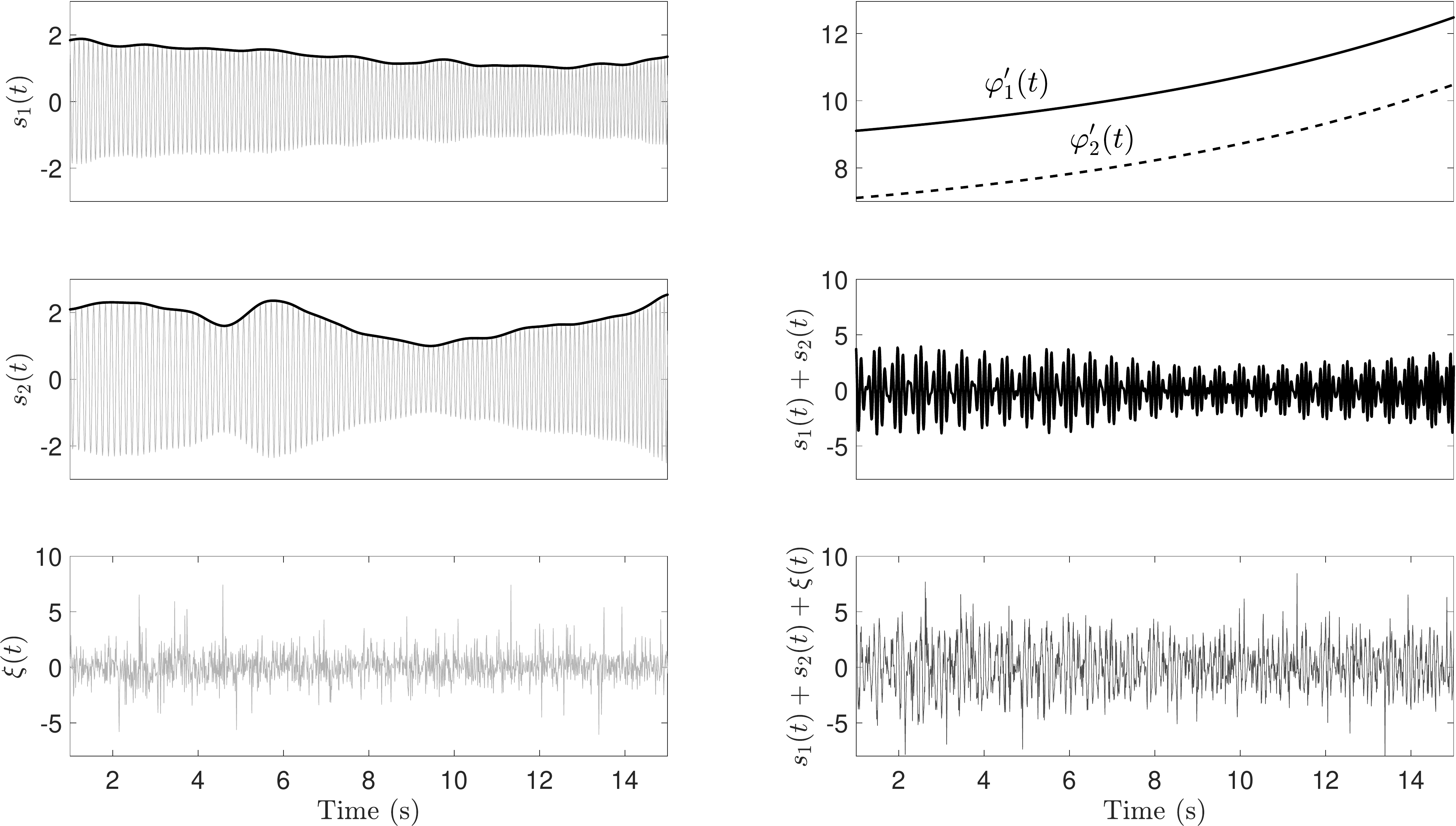}
\caption{Example 1 and 2: Left column from top to bottom: first and second simple component and noise. Right column from top to bottom: instantaneous frequencies; clean signal; signal with additive noise }
\label{fig:Ex2_wNoise_signal}
\end{figure}

In Figure \ref{fig:Ex1_noNoise_if} we show the TFR produced by the SST superimposed with the derived instantaneous frequency curves.

\begin{figure}[ht]
\centering
\includegraphics[width=0.5\linewidth]{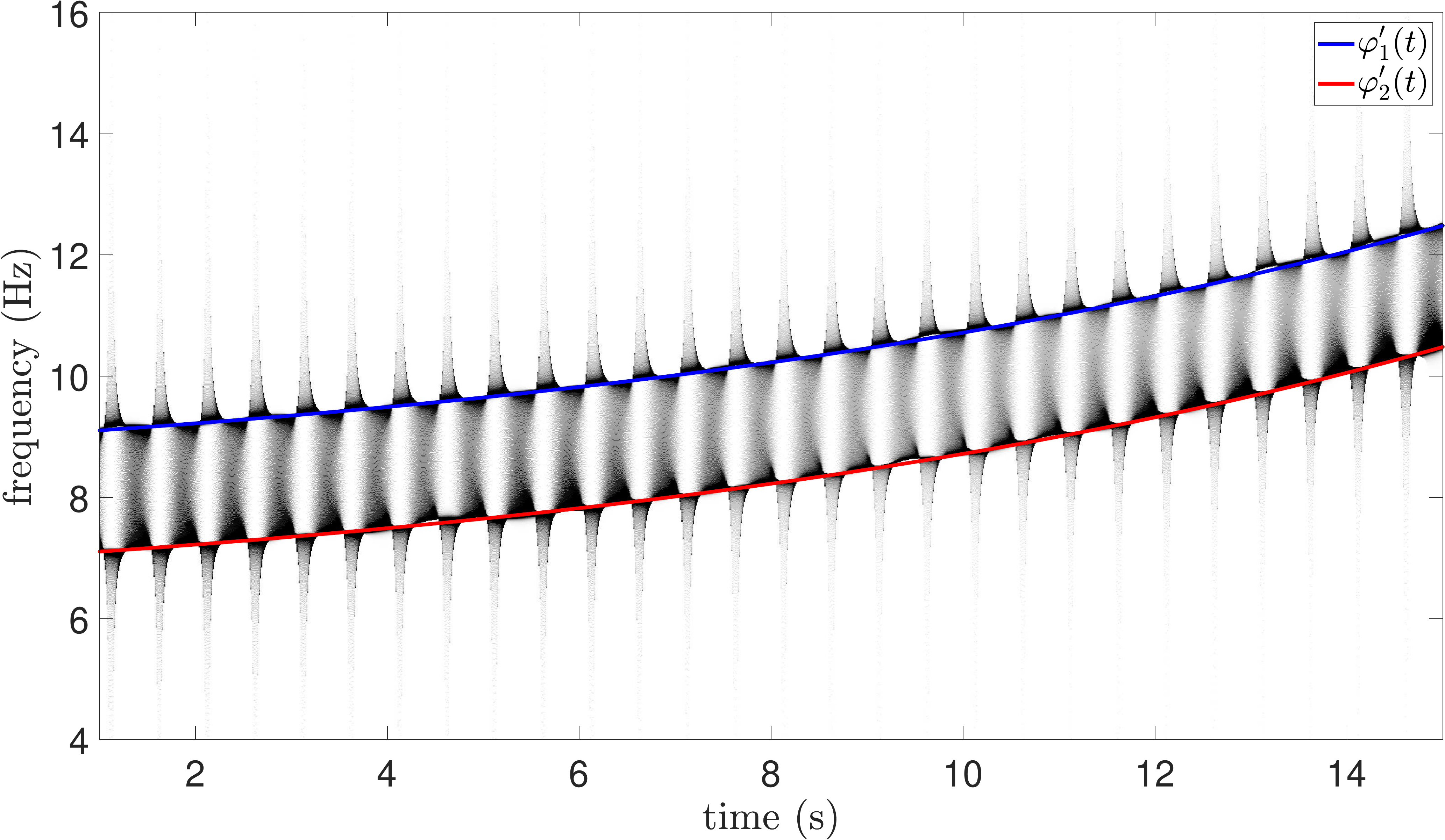}
\caption{Example 1: Instantaneous frequencies curves on top of the TFR plot }
\label{fig:Ex1_noNoise_if}
\end{figure}

In Figure \ref{fig:Ex1_noNoise_IMTs} we plot the IMTs produced by SIFT and BPF algorithms, as well as the ones obtained using the SST reconstruction formula \eqref{eq:SST_recon_formula2}. Whereas in Figure \ref{fig:Ex1_noNoise_diff}  and in Table \ref{tab:ex1} we report the differences between the ground truth and the different IMTs. From these comparisons it is evident that the SST reconstruction without any window size tuning performs poorly. After tuning the window length to 677, the SST performance are clearly improved. The proposed SIFT algorithm, instead, allows to obtain a good performance without any tuning.

\begin{figure}[ht]
\centering
\includegraphics[width=\linewidth]{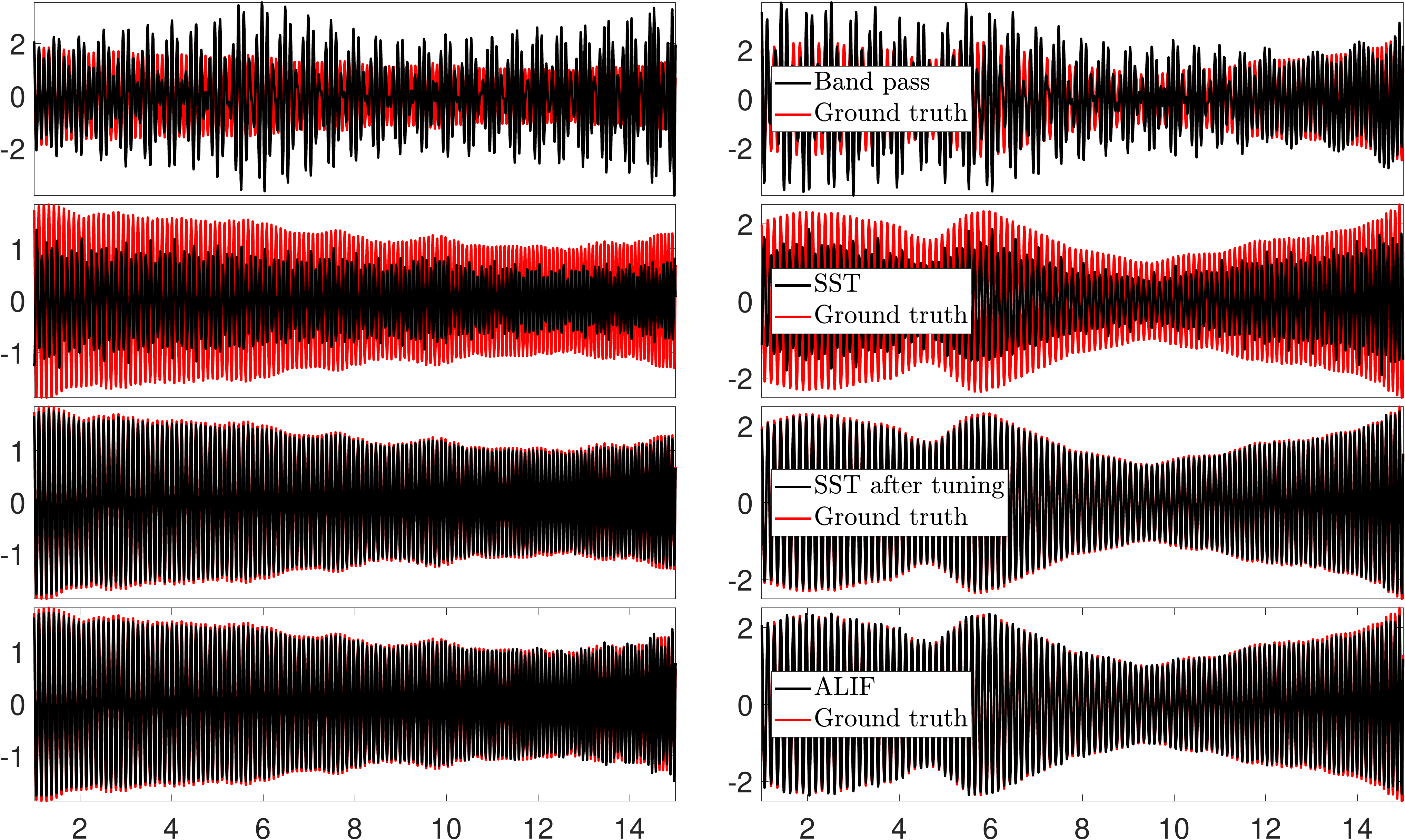}
\caption{Example 1: the ground truth and the BPF, SST, without or with window tuning, and SIFT first and second IMTs, left and right panel, respectively}
\label{fig:Ex1_noNoise_IMTs}
\end{figure}

\begin{figure}[ht]
\centering
\includegraphics[width=0.48\linewidth]{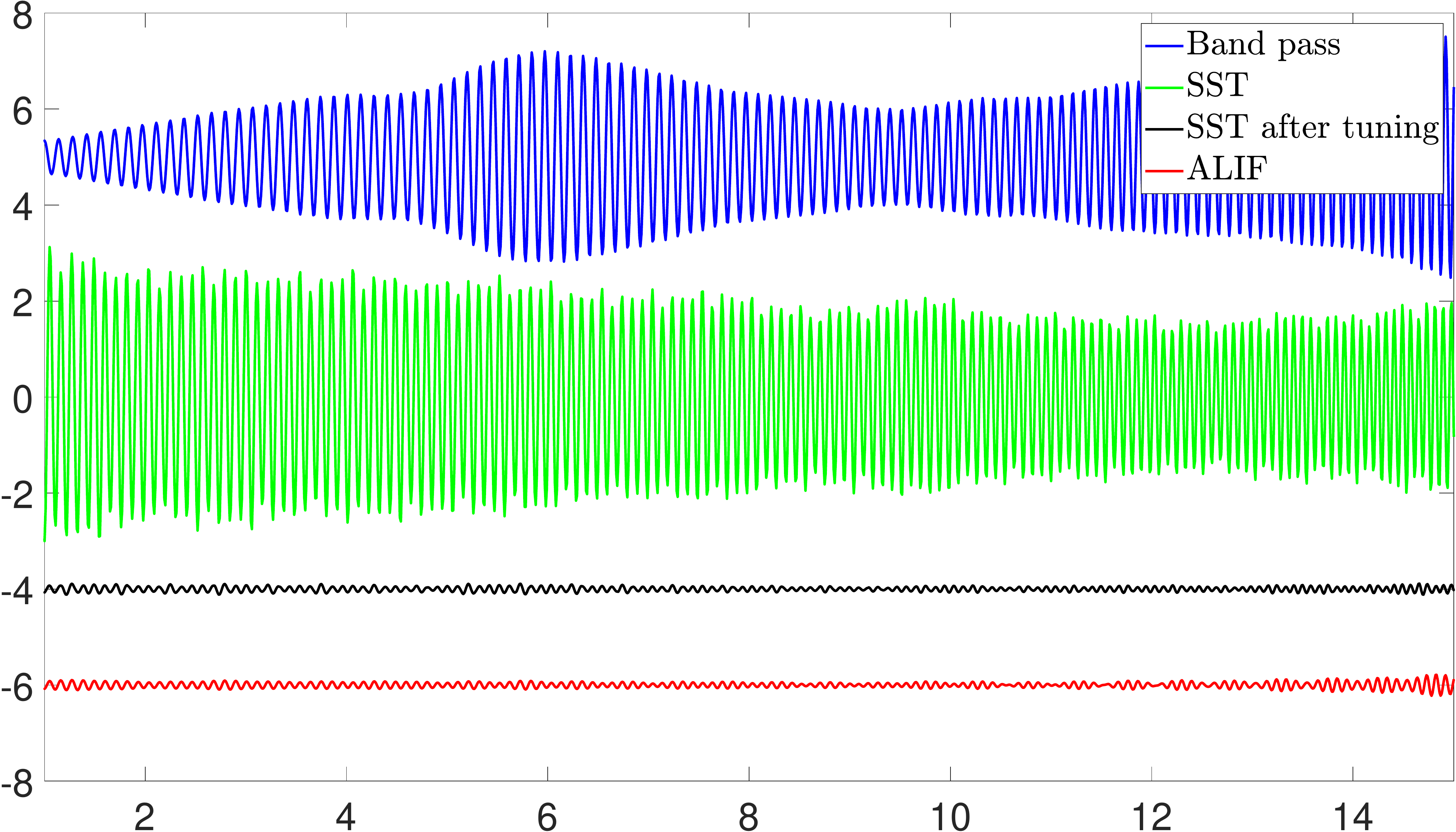}
\includegraphics[width=0.48\linewidth]{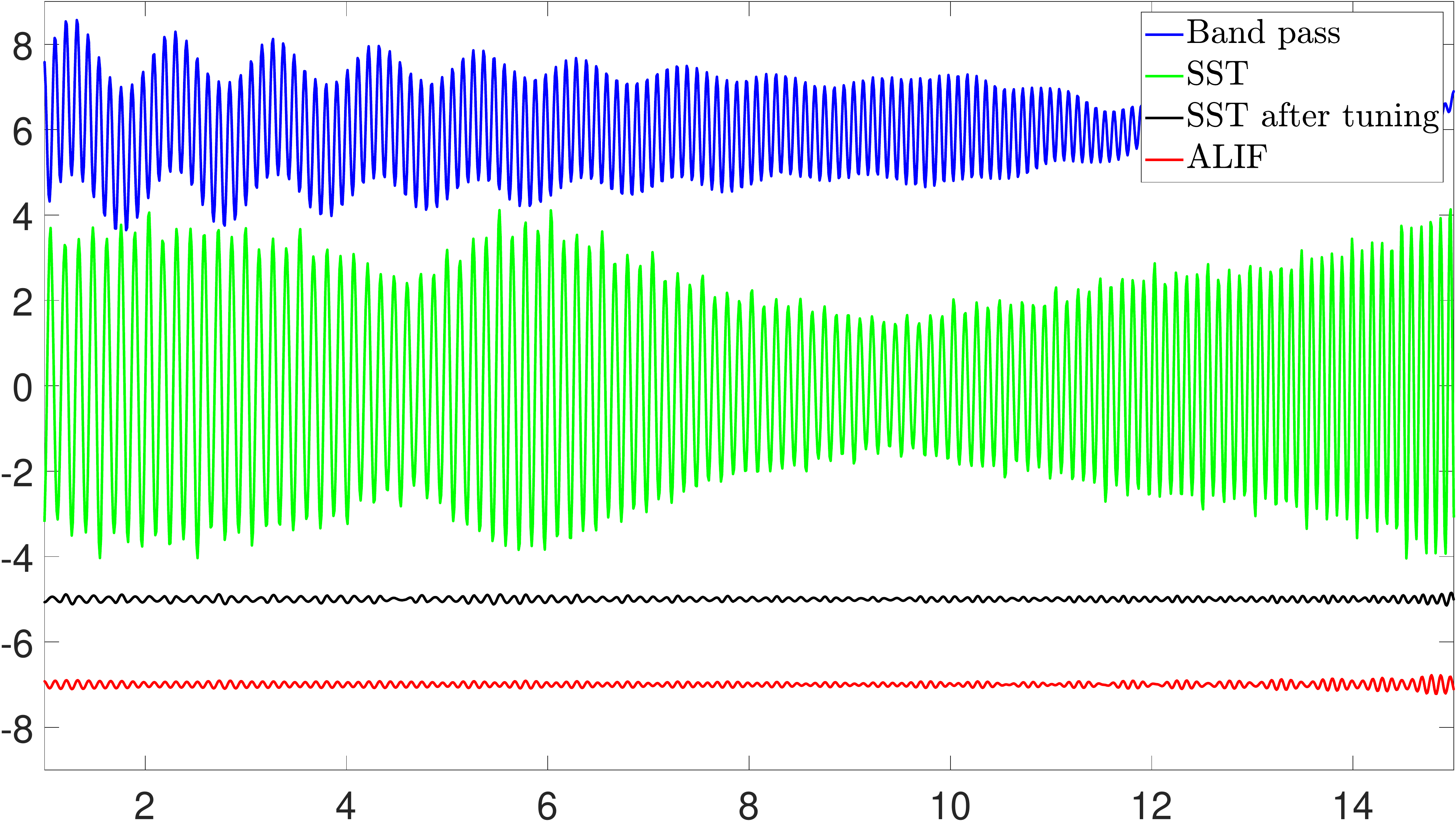}
\caption{Example 1. Differences between the ground truth and the BPF, SST, without or with a tuning, and SIFT first and second IMTs, left and right panel, respectively. We shift vertically these curves to increase their visibility}
\label{fig:Ex1_noNoise_diff}
\end{figure}

\begin{table}
  \centering
\begin{tabular}{|c|c|c|c|c|}
  \hline
      & BPF & SST & SST after tuning  & SIFT \\
       \hline\hline
  $\textrm{IMT}_1$ & $1.0972$ & $1.3786$ & $0.1213$  & $0.2080$ \\
  \hline
  $\textrm{IMT}_2$ & $0.7664$ & $1.4780$ & $0.0939$  & $0.1562$ \\
  \hline
\end{tabular}
\caption{Example 1: Relative error in norm 2 of the different decompositions with respect to the ground truth}\label{tab:ex1}
\end{table}


\subsection{Example 2}

We consider now the signal of Example 1 with additive noise where the Signal to Noise Ratio expressed in decibels ($\textrm{SNR}_{\textrm{dB}}$) is  around $1.7$. This is computed as $$\textrm{SNR}_{\textrm{dB}}=20\log(\|\textrm{signal}\|_2/\|\textrm{noise}\|_2)$$
In bottom row of Figure \ref{fig:Ex2_wNoise_signal} we report the noise and the corresponding signal.

The TFR plot, Figure \ref{fig:Ex2_wNoise_if} compared with Figure \ref{fig:Ex1_noNoise_if}, it is now deteriorated by noise.

\begin{figure}[ht]
\centering
\includegraphics[width=0.5\linewidth]{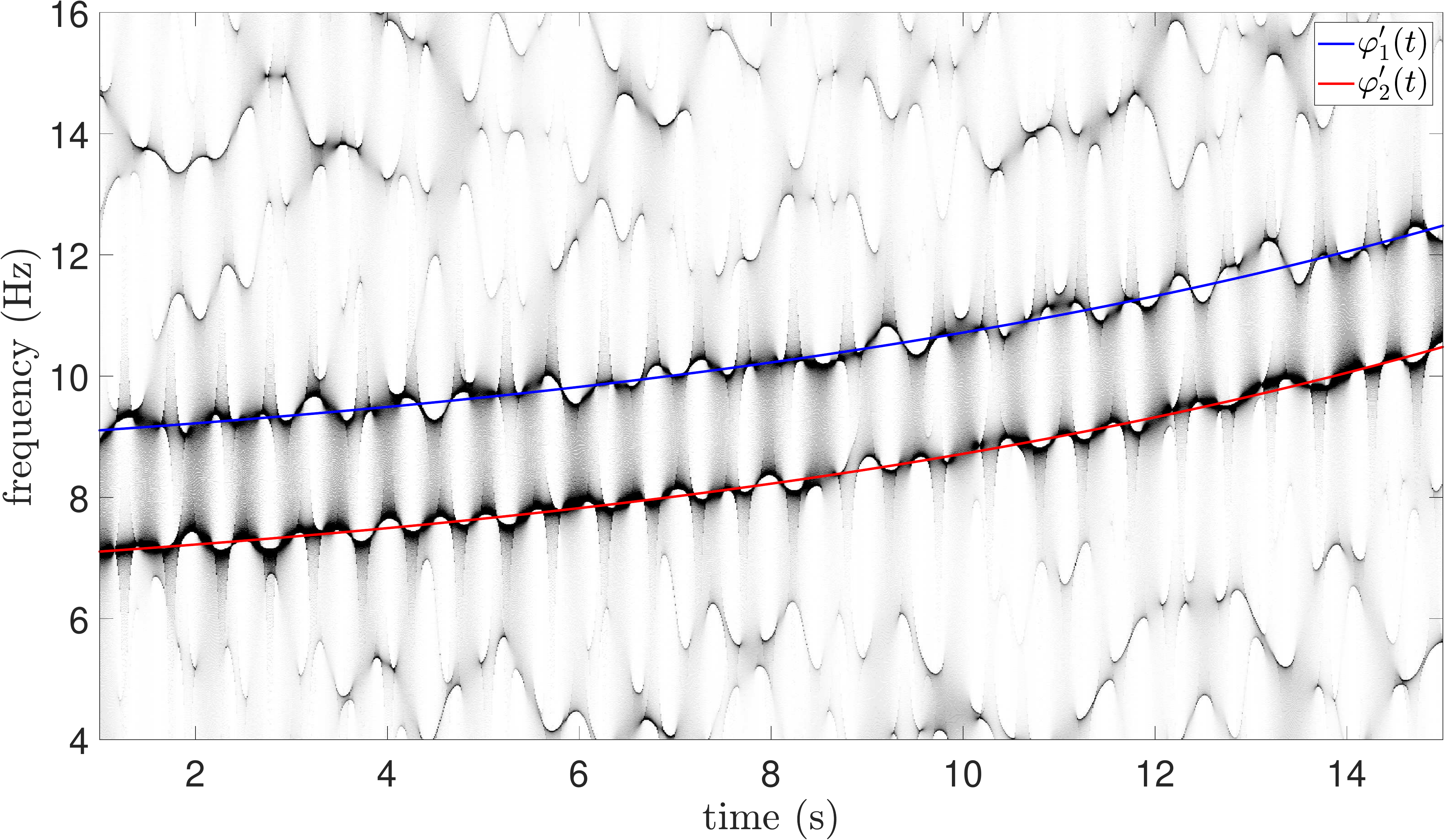}
\caption{Example 2: Instantaneous frequencies curves on top to the TFR plot }
\label{fig:Ex2_wNoise_if}
\end{figure}

As for the noiseless signal, also in this context the SST algorithm requires a window tuning, where the window length is tuned now to 677, in order to achieve good performance, whereas SIFT does not require it, Figures \ref{fig:Ex2_wNoise_IMTs} and \ref{fig:Ex2_wNoise_diff} and Table \ref{tab:ex2}. To better test the performance of SIFT with respect to BPF and SST methods, we consider 100 different noise realizations which we add to the original noiseless signal. The average errors, measured in norm 2, and their standard deviations are reported in Table \ref{tab:ex2}.

\begin{figure}[ht]
\centering
\includegraphics[width=\linewidth]{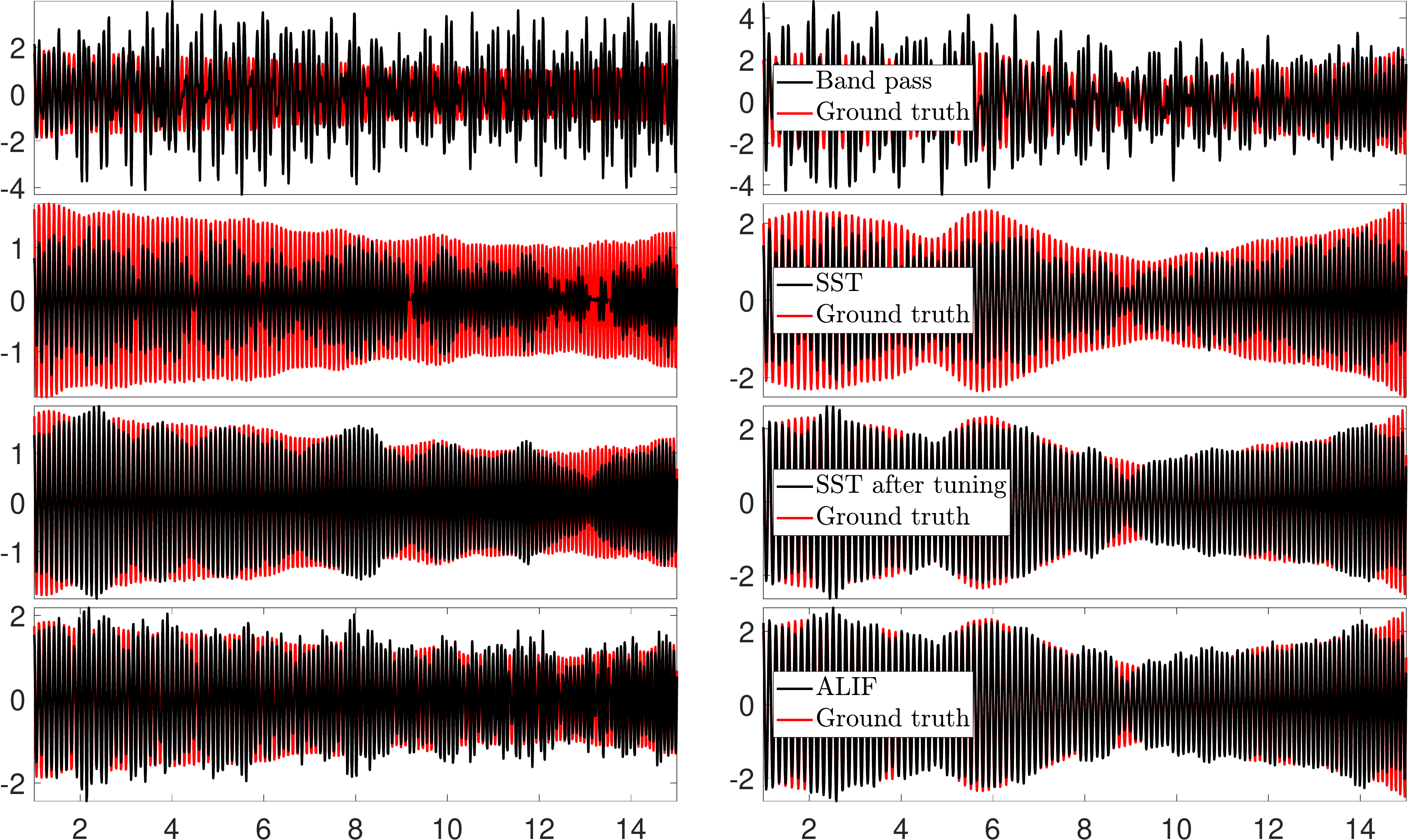}
\caption{Example 2: the ground truth and the BPF, SST, without or with window tuning, and SIFT first and second IMTs, left and right panel, respectively}
\label{fig:Ex2_wNoise_IMTs}
\end{figure}

\begin{figure}[ht]
\centering
\includegraphics[width=0.48\linewidth]{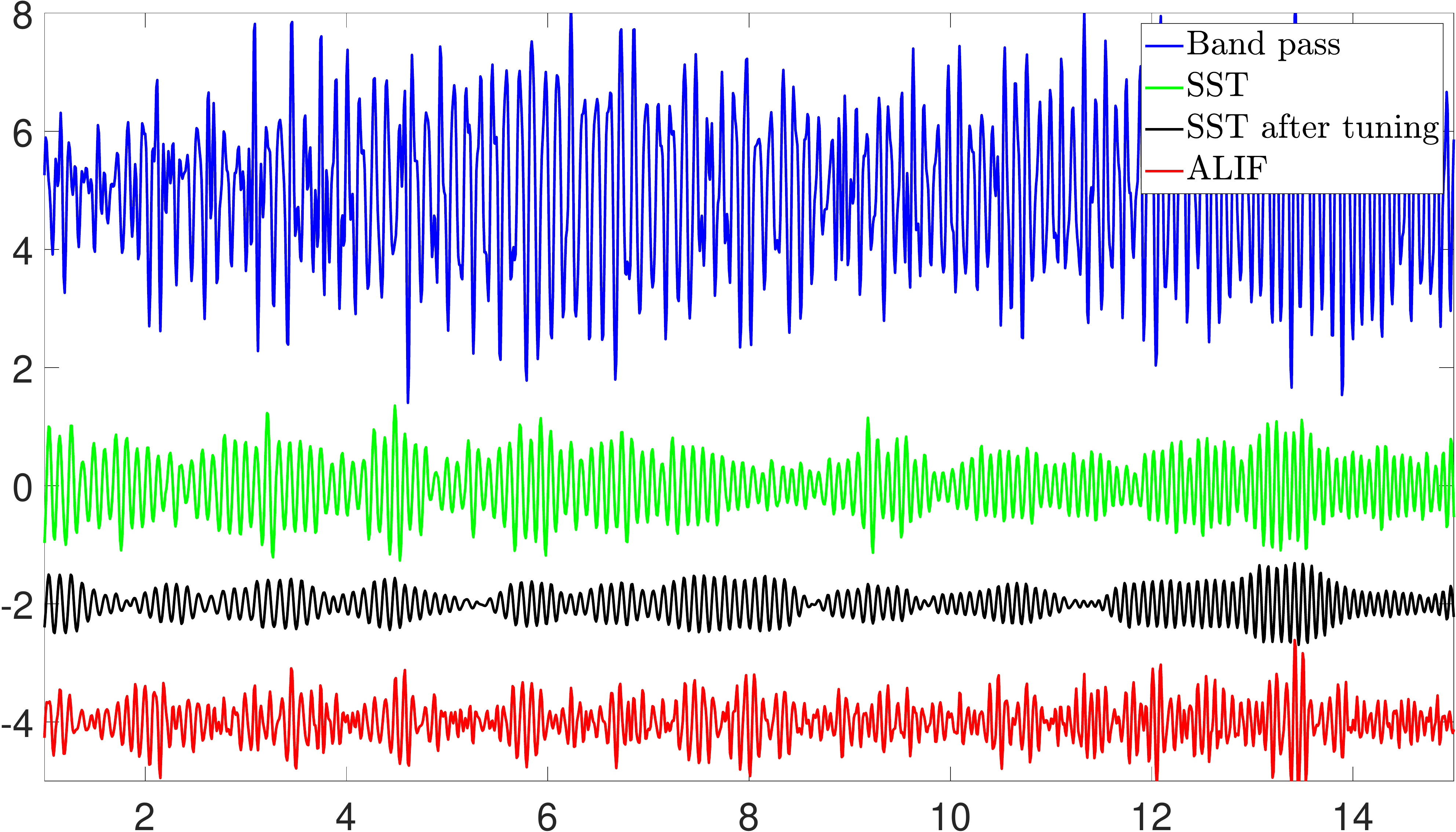}
\includegraphics[width=0.48\linewidth]{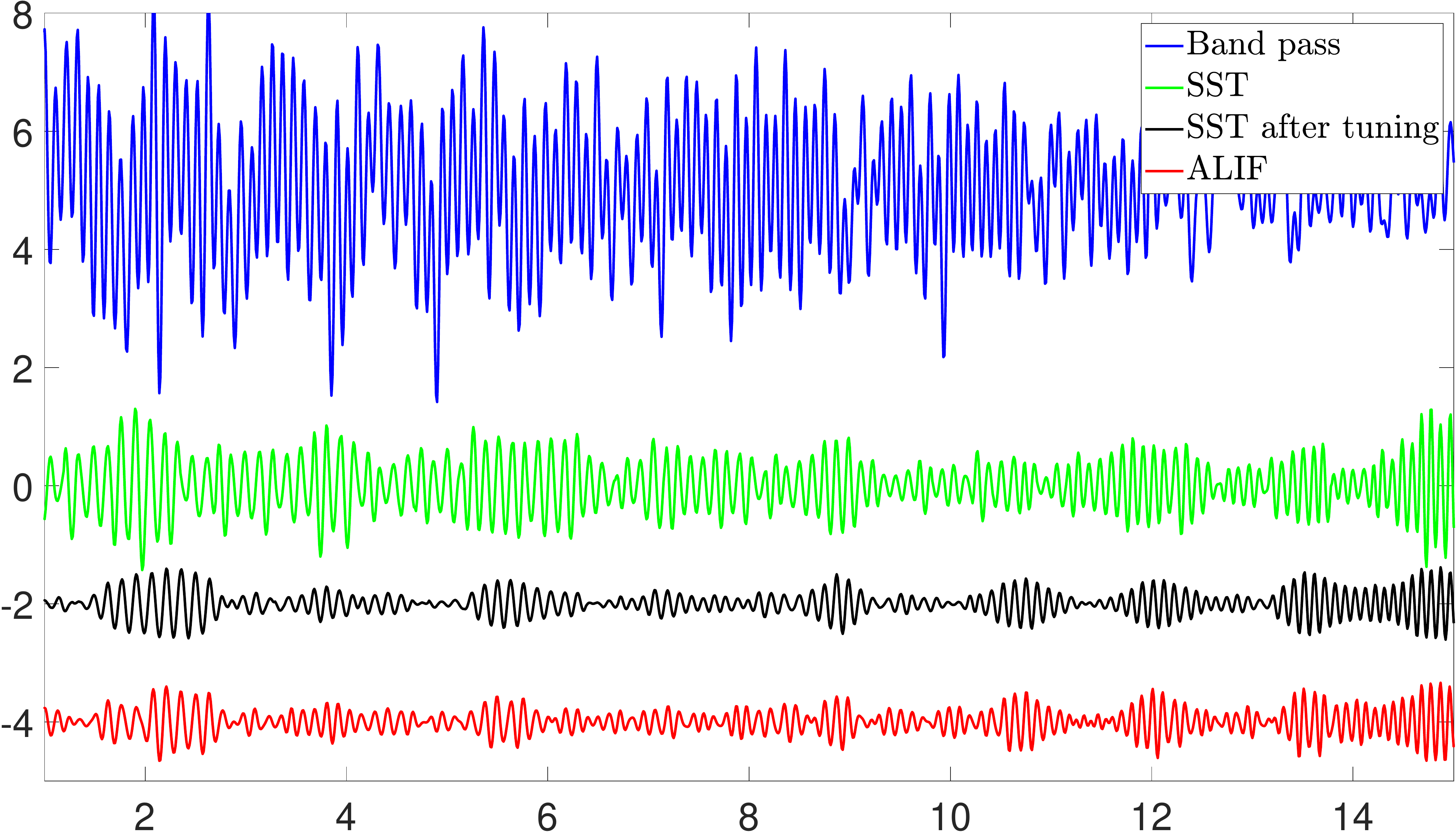}
\caption{Example 2: Differences between the ground truth and the BPF, SST, without or with a tuning, and SIFT first and second IMTs, left and right panel, respectively. We shift vertically these curves to increase their visibility}
\label{fig:Ex2_wNoise_diff}
\end{figure}

\begin{table}
  \centering
\begin{tabular}{|c|c|c|c|c|}
  \hline
      & BPF & SST & SST after tuning  & SIFT \\
       \hline\hline
  $\textrm{IMT}_1$ & $0.6501$ $(0.0197)$ & $0.3342$ $(0.0157)$ & $0.1594$ $(0.0131)$  & $0.2378$ $(0.0127)$ \\
  \hline
  $\textrm{IMT}_2$ & $1.4348$ $(0.0379)$ & $0.4856$ $(0.0273)$ & $0.2195$ $(0.0228)$  & $0.2337$ $(0.0211)$ \\
  \hline
\end{tabular}
\caption{Example 2: Mean and standard deviation, respectively outside and inside parentheses, of the relative errors in norm 2 of the different decompositions with respect to the ground truth for 100 of different noise realizations}\label{tab:ex2}
\end{table}


\subsection{Example 3}

In this example we consider another nonstationary signal whose components have non-overlapping iF's, but the ranges of the iF's overlap. This feature creates problem in the signal decomposition with standard methods like Short Time Fourier and wavelet Transform, as well as Bandpass filter.
In Figure \ref{fig:Ex4_signal} we show the simple components $s_1$ and $s_2$, their instantaneous frequencies and one realization of the additive noise. The signal shown in the bottom row of that figure as $\textrm{SNR}_{\textrm{dB}}~=~1.93$.

\begin{figure}[ht]
\centering
\includegraphics[width=0.9\linewidth]{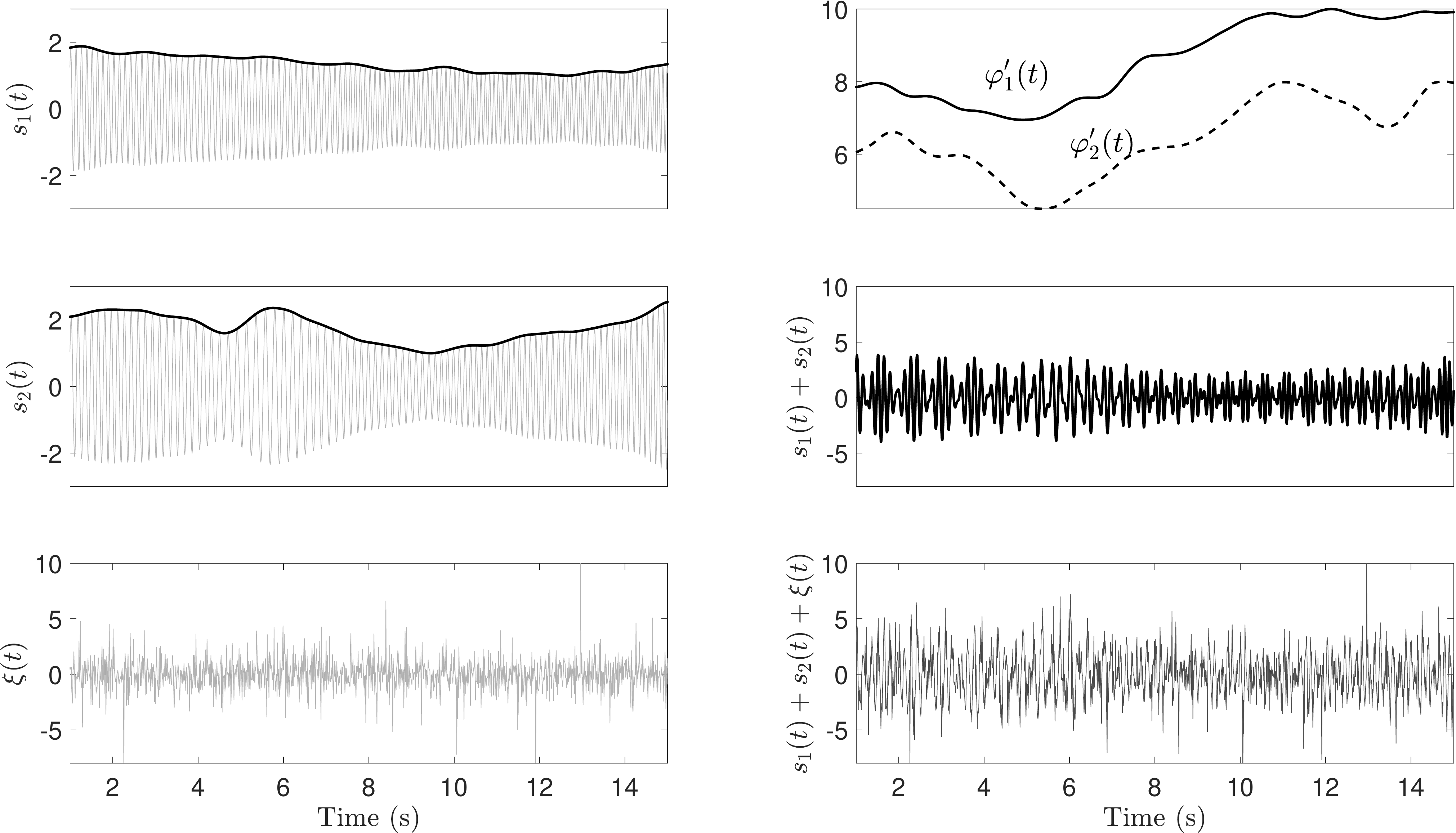}
\caption{Example 3 signal. Left column from top to bottom: first and second simple component and noise. Right column from top to bottom: instantaneous frequencies; clean signal; signal with additive noise}
\label{fig:Ex4_signal}
\end{figure}

Using SIFT approach we obtain the TFR and the instantaneous frequency curves shown in the left panel of Figure \ref{fig:Ex4_if}.

In Figures \ref{fig:Ex4_IMTs} and \ref{fig:Ex4_diff} we plot the IMTs produced using different approaches and their differences with respect to the ground truth, when applied to the signal shown in Figure \ref{fig:Ex4_signal}. The tuned SST window length is equal to a Gaussian window supported on $[-6,6]$ second discretized into $677$ equally spaced sample points.

\begin{figure}[ht]
\centering
\includegraphics[width=\linewidth]{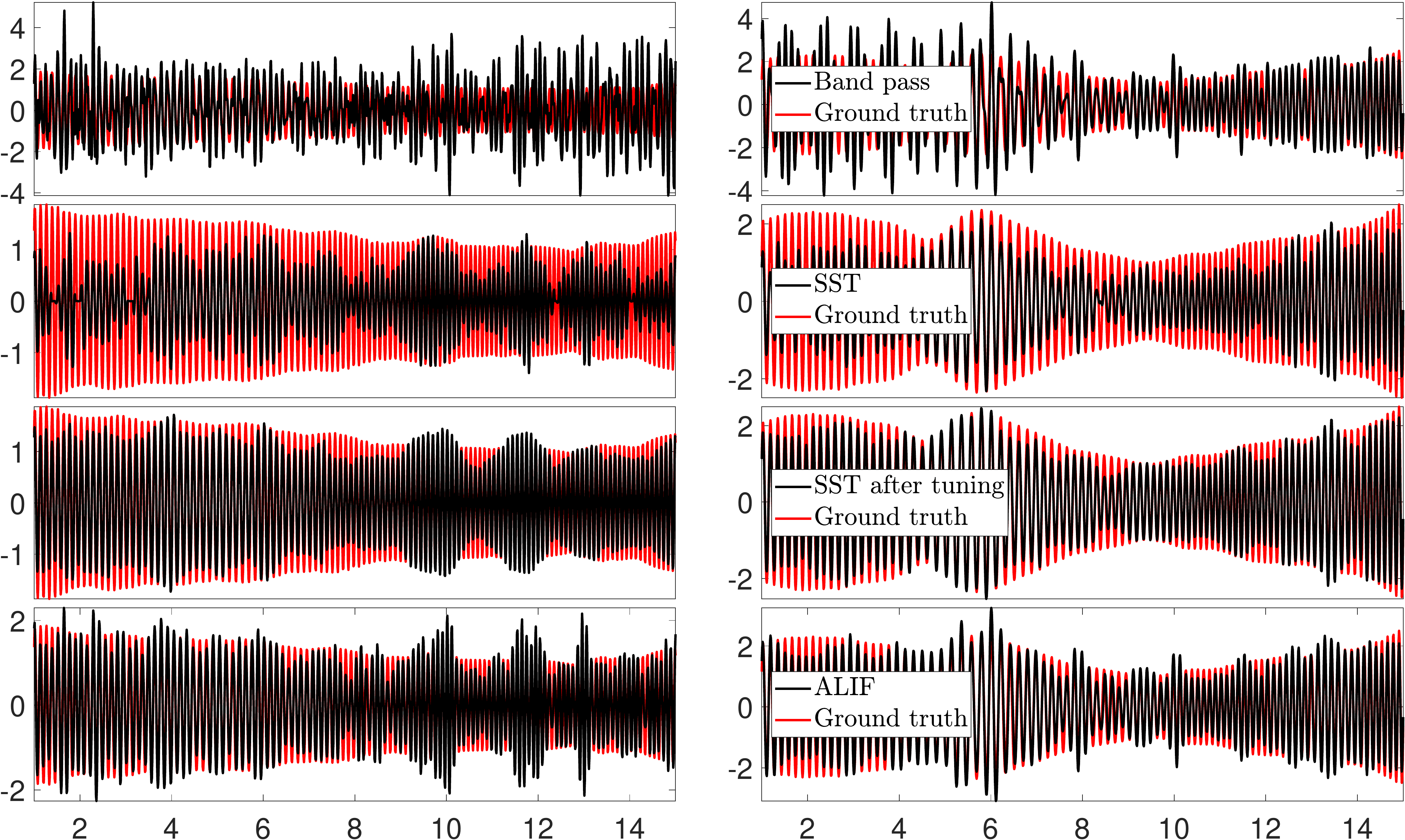}
\caption{Example 3: the ground truth and the BPF, SST, without or with window tuning, and SIFT first and second IMTs, left and right panel, respectively}
\label{fig:Ex4_IMTs}
\end{figure}

\begin{figure}[ht]
\centering
\includegraphics[width=0.48\linewidth]{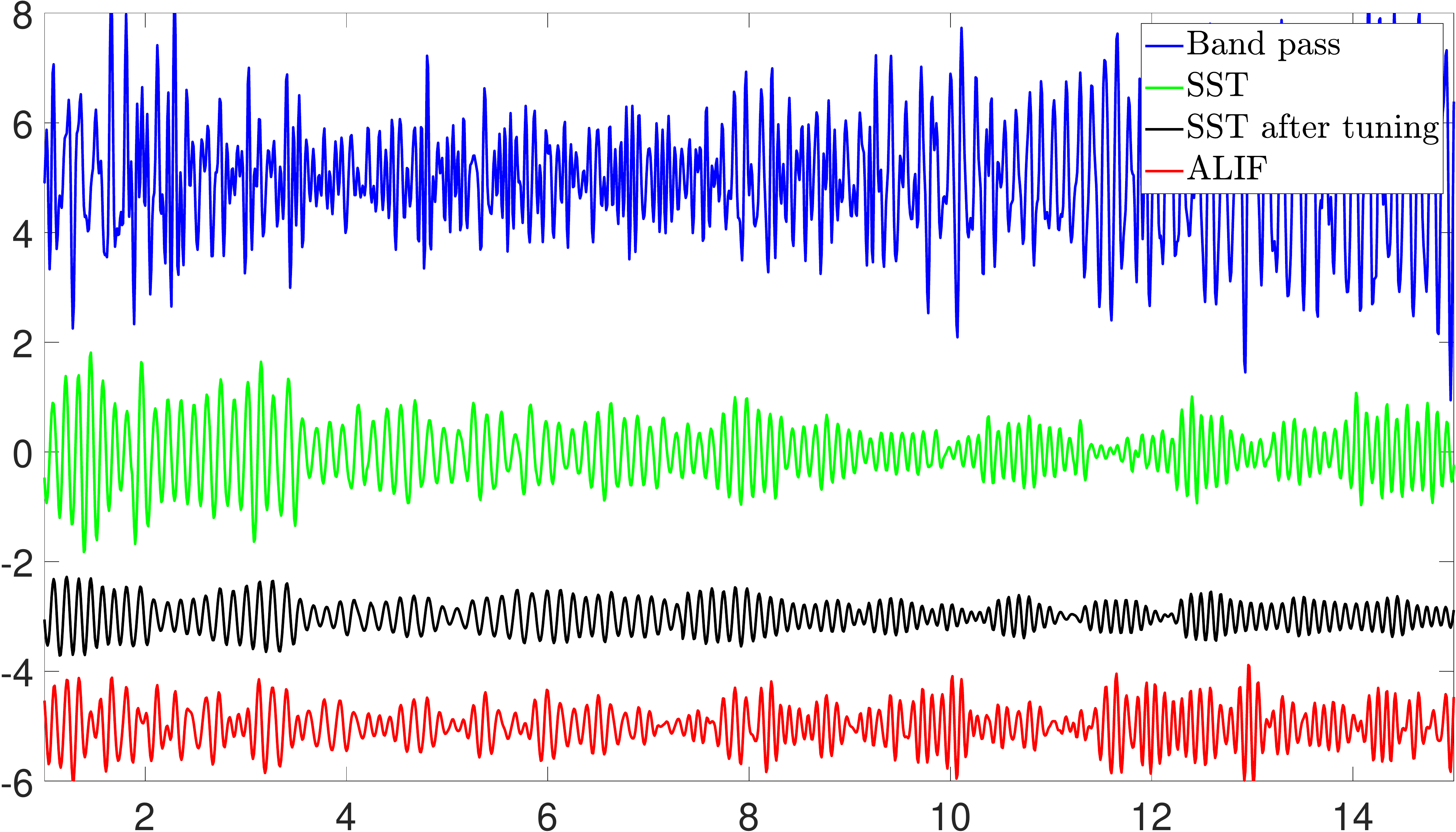}
\includegraphics[width=0.48\linewidth]{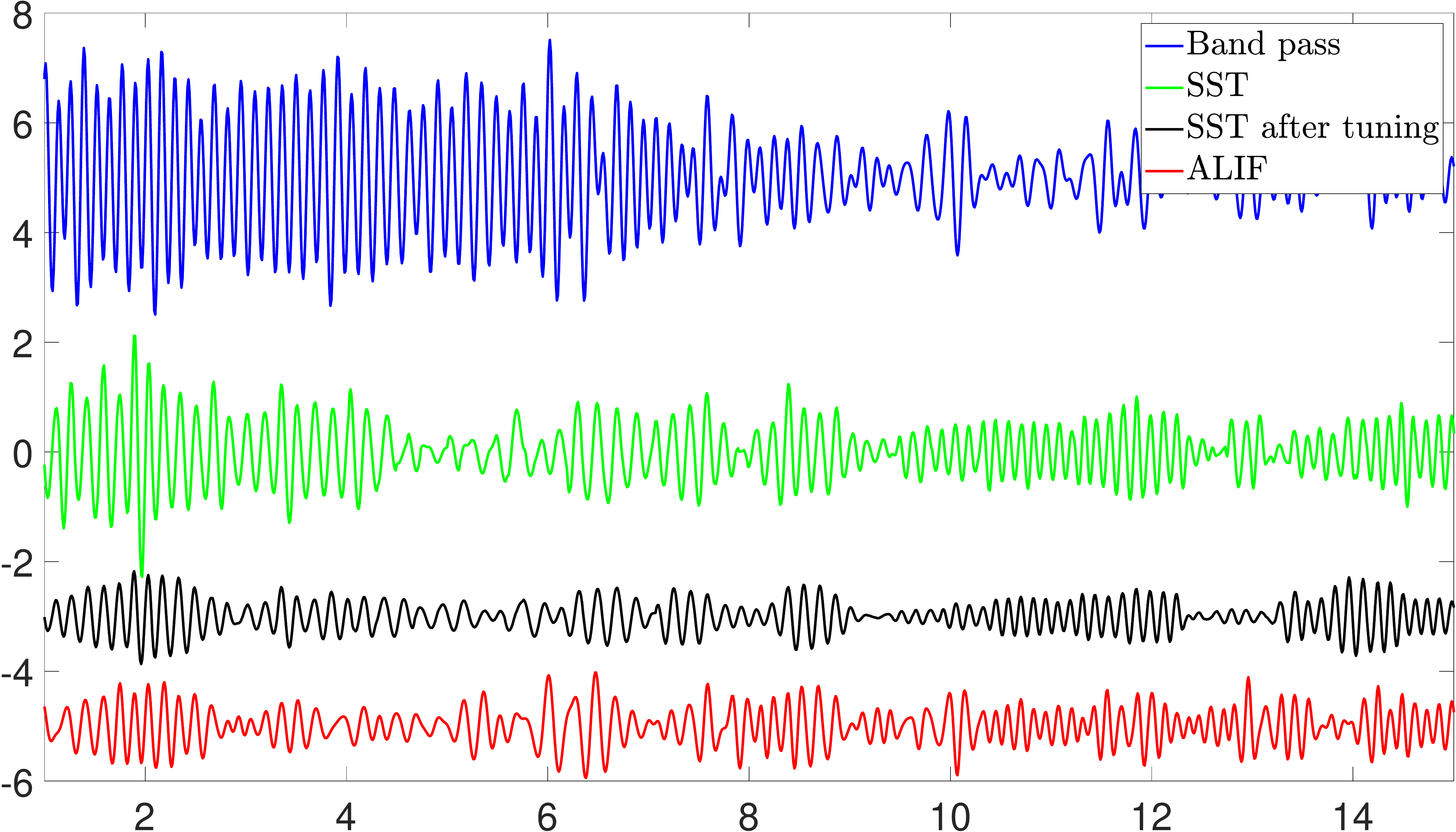}
\caption{Example 3: Differences between the ground truth and the BPF, SST, without or with a tuning, and SIFT first and second IMTs, left and right panel, respectively. We shift vertically these curves to increase their visibility}
\label{fig:Ex4_diff}
\end{figure}

In Figure \ref{fig:Ex4_if} right panel, we show the TFR produced by SST, and the proposed SIFT-TFR. Clearly, we see that the iF's represented in the SIFT-TFR are crispier, with less wiggle, and more concentrated around the ideal iF values. Moreover, the noisy pattern in the background is less dominant, and the spectral leakage before 4-th second in the TFR determined by SST is less dominate.
\begin{figure}[ht]
\centering
\includegraphics[width=0.48\linewidth]{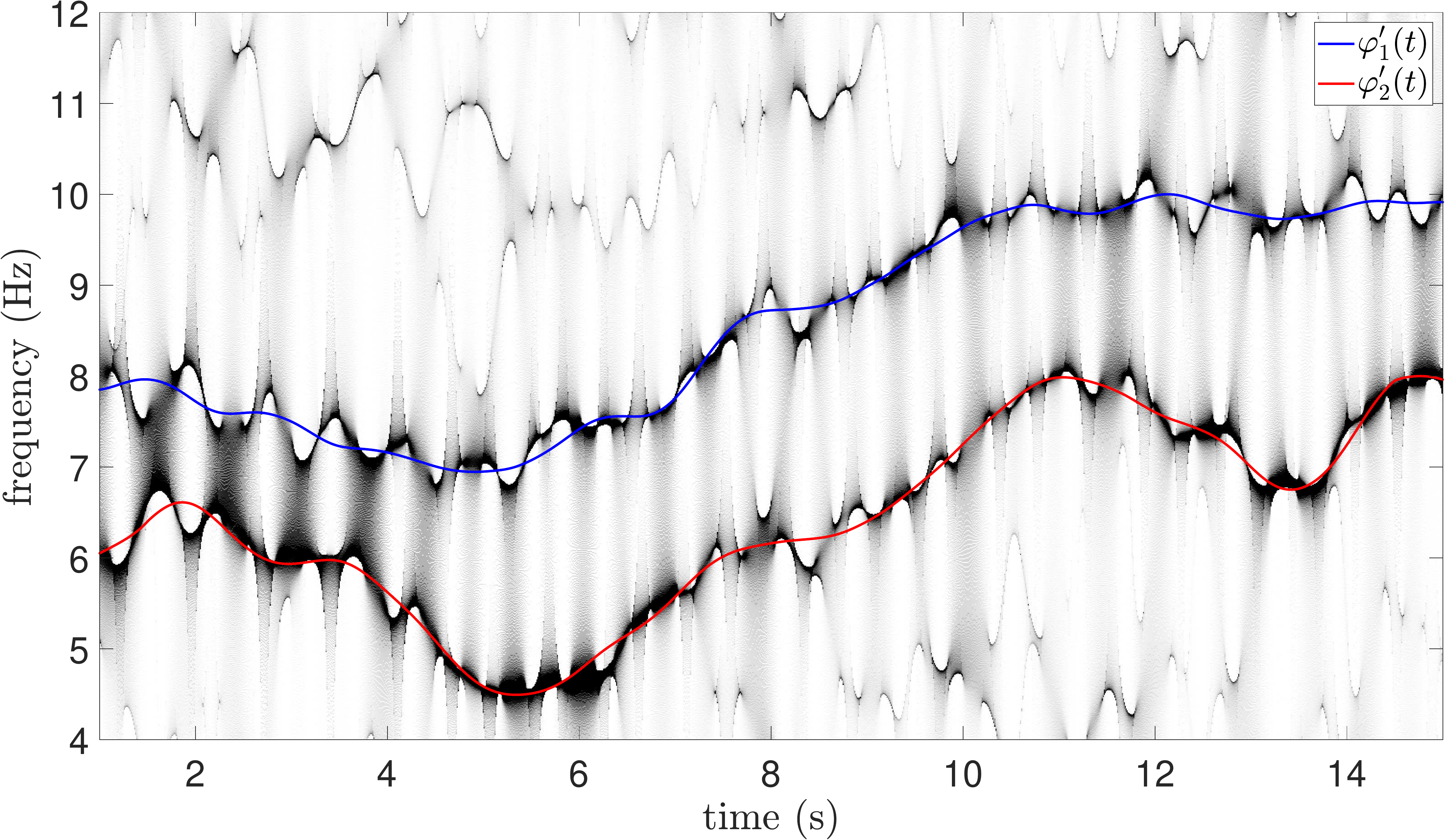}
\includegraphics[width=0.48\linewidth]{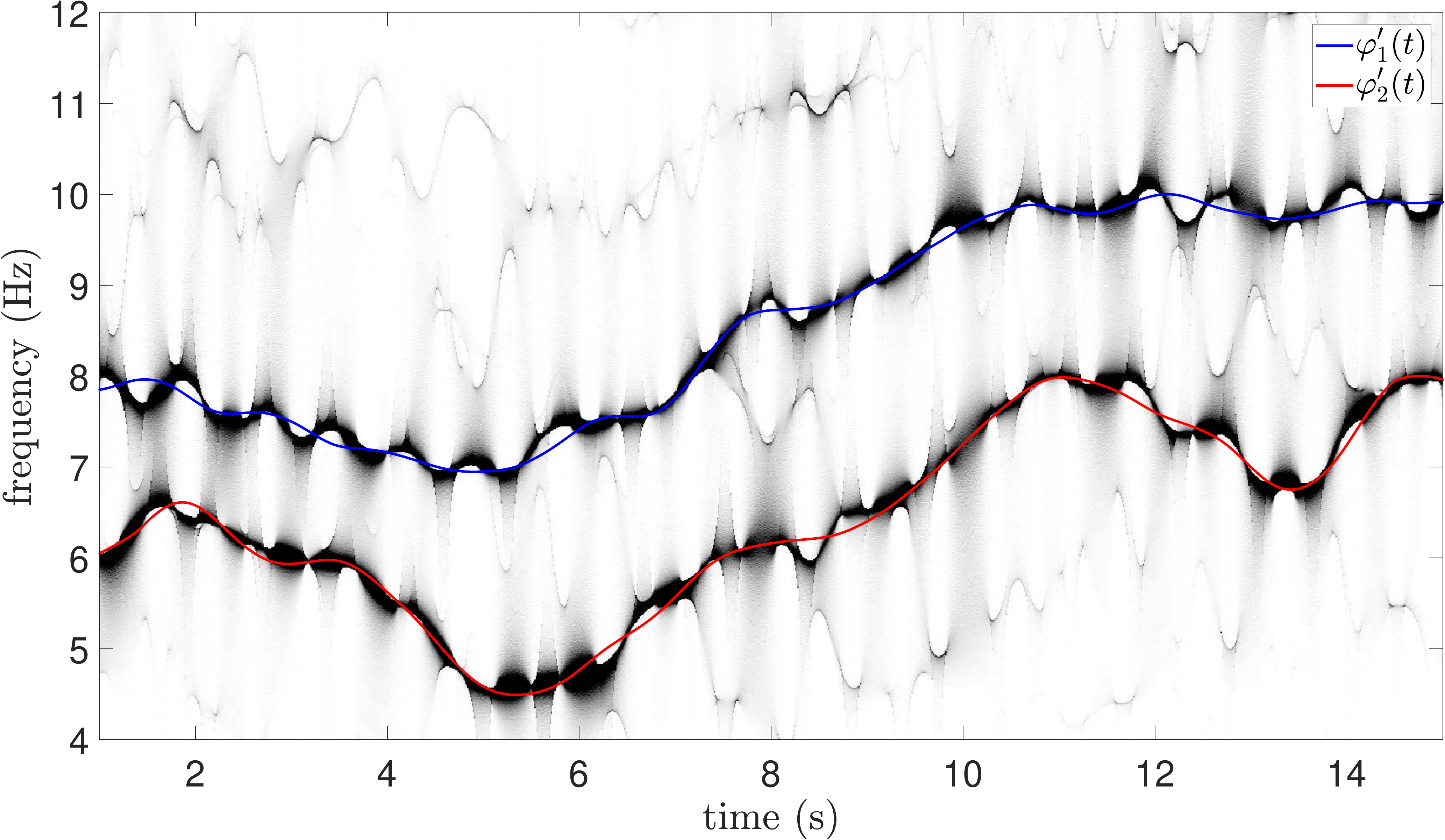}
\caption{Example 3: Left panel, instantaneous frequencies curves on top of the TFR plot of the original signal. Right panel, TFR of the SIFT IMTs}
\label{fig:Ex4_if}
\end{figure}

In Table \ref{tab:Ex4} we report the mean and standard deviation values of the norm 2 differences between the ground truth and the IMTs produced using SIFT, BPF ad SST methods when applied to the signal presented in Figure \ref{fig:Ex4_signal} with 100 different noise realizations. Also in this case SIFT outperform both BPF and SST without a window tuning.

\begin{table}
  \centering
\begin{tabular}{|c|c|c|c|c|}
  \hline
      & BPF & SST & SST after tuning  & SIFT \\
       \hline\hline
  $\textrm{IMT}_1$ & $0.6501$ $(0.0197)$ & $0.3342$ $(0.0157)$ & $0.1594$ $(0.0131)$  & $0.2378$ $(0.0127)$ \\
  \hline
  $\textrm{IMT}_2$ & $1.4348$ $(0.0379)$ & $0.4856$ $(0.0273)$ & $0.2195$ $(0.0228)$  & $0.2337$ $(0.0211)$ \\
  \hline
\end{tabular}
\caption{Example 3:  Mean and standard deviation, respectively outside and inside parentheses, of the relative errors in norm 2 of the different decompositions with respect to the ground truth for 100 of different noise realizations}\label{tab:Ex4}
\end{table}


\subsection{Example 4}

In this example to test the ability of SIFT method in separating two components with crossing instantaneous frequencies. We show the signal under analysis in Figure \ref{fig:Ex3_signal} included one realization of noise. The shown signal as $\textrm{SNR}_{\textrm{dB}}~=~3.13$.

\begin{figure}[ht]
\centering
\includegraphics[width=0.9\linewidth]{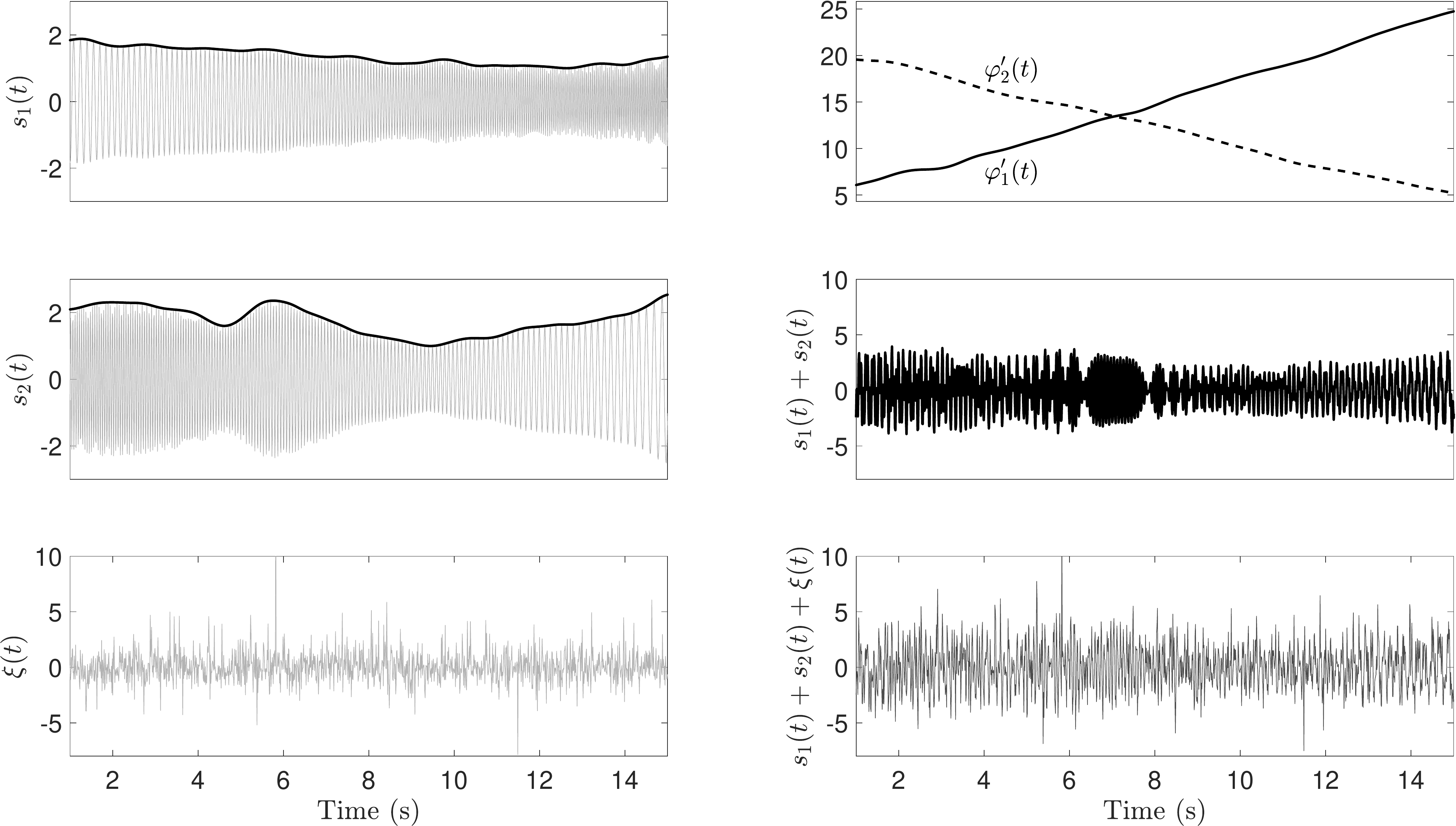}
\caption{Example 4: Left column from top to bottom: first and second simple component and noise. Right column from top to bottom: instantaneous frequencies; clean signal; signal with additive noise. $\textrm{SNR}_{\textrm{dB}} = 3.13$ }
\label{fig:Ex3_signal}
\end{figure}

\begin{figure}[ht]
\centering
\includegraphics[width=0.48\linewidth]{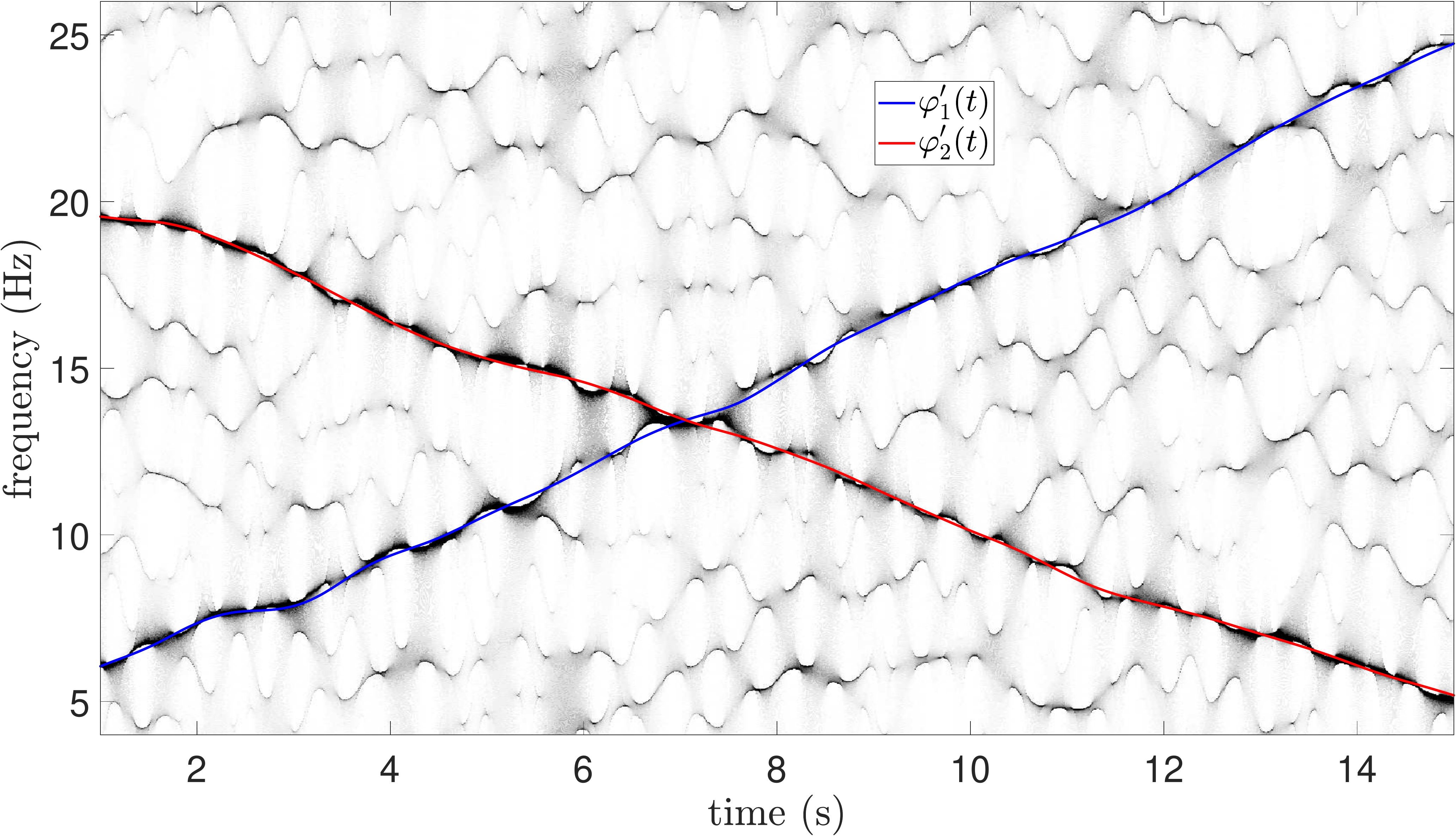}
\includegraphics[width=0.48\linewidth]{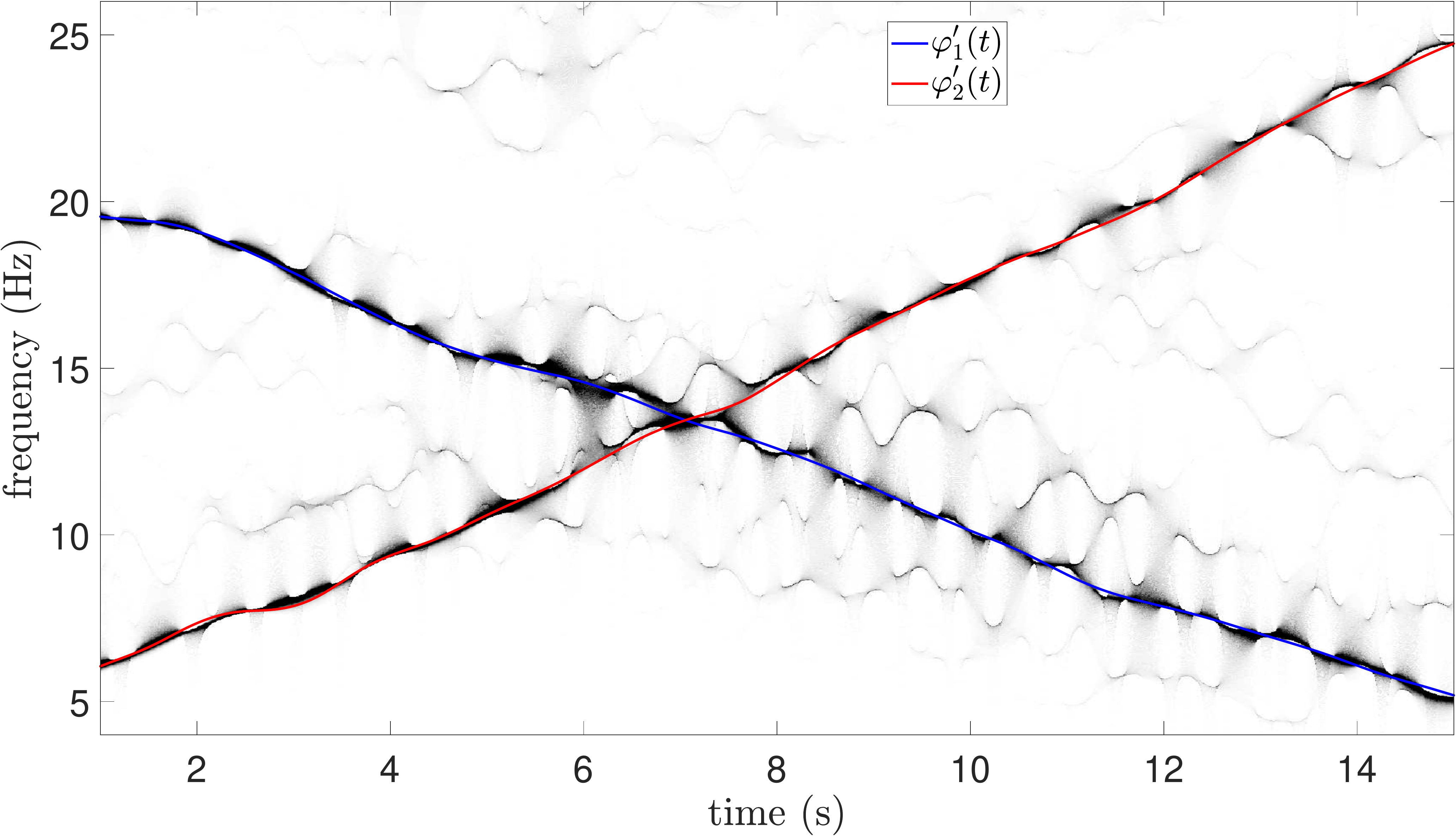}

\caption{Example 4: Left panel, instantaneous frequencies curves on top of the TFR plot of the original signal. Right panel, TFR of the SIFT IMTs}
\label{fig:Ex3_if}
\end{figure}

In order to test SIFT performance in extracting the two components, we consider 100 different realizations of noise. The mean relative differences, measured in norm 2, between the ground truth and the IMTs produced by SIFT out of 100 noise realizations are $0.5677$, with standard deviation of $0.0290$, and $0.4191$, with standard deviation of $0.0191$, for the first and second IMT respectively.

\begin{figure}[ht]
\centering
\includegraphics[width=0.48\linewidth]{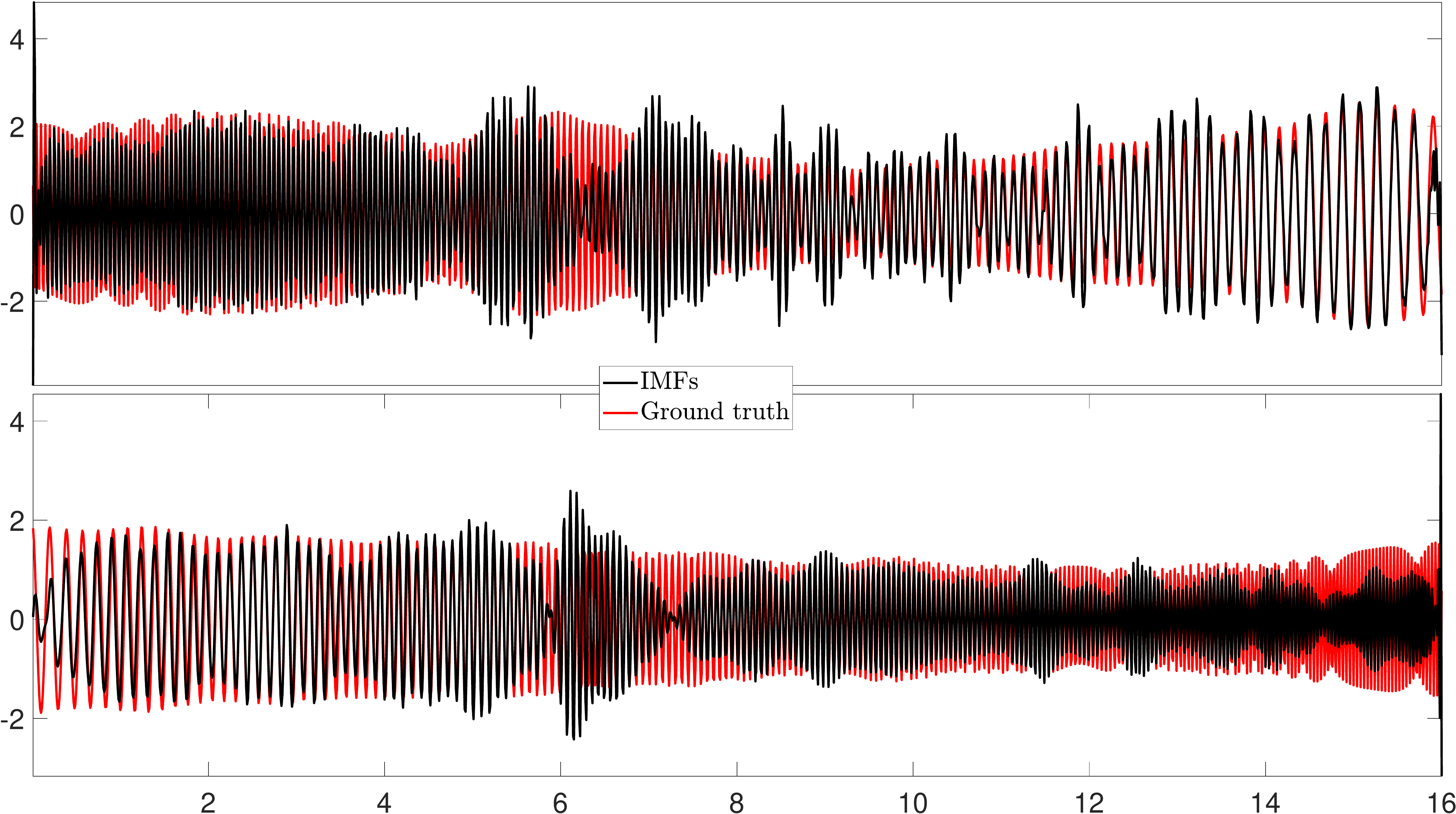}
\caption{Example 4: SIFT decomposition compared with the ground truth}
\label{fig:Ex3_IMTs}
\end{figure}
In Figure \ref{fig:Ex3_if} right panel, we show the TFR produced by SST, and the proposed SIFT-TFR. Also in this case we see that the iF's represented in the SIFT-TFR are crispier, with less wiggle, and more concentrated around the ideal iF values. The noisy pattern in the background is drastically reduced.


\section{Discussion and Conclusion}\label{sec:Conclusions}

In this work we study the convergence of the Adaptive Local Iterative Filtering (ALIF) method and we propose a new approach for the decomposition and analysis of highly nonstationary signals called Synchrosqueezing Iterative Filtering Technique (SIFT), which leverages on the ALIF and SST features, to produce a stable, convergent and robust method which allows to reconstruct the single highly nonstationary components contained in a signal.

Regarding the ALIF convergence, we were able to prove that the method converges and extracts an IMT from a given signal.  We plan to extend in the future the results presented in this work to the case of the general operator given in \eqref{eq:K_gen}. Furthermore in this paper we do not analyze how ALIF behaves when the signal is noisy. Indeed, suppose the input signal $f$ is contaminated by a Gaussian white noise $\Phi$, what is $\mathcal{S}_\sigma^{(K)}(f+\Phi)$ as a random process, and how is it related to $\mathcal{S}_\sigma^{(K)}f$? This kind of statistical problem, particularly for the nonlinear-type time-frequency approach, was only partially studied, except some recent efforts \cite{sourisseau2019}. We will explore this problem in a future work, since it is the key step toward the statistical inference of a signal features.

Concerning the SIFT algorithm, we showed numerically that it performs better than other signal decomposition algorithms, including SST and BPF reconstruction, particularly when the signal is contaminated by noise, in the following sense.
Note that from the results shown in the Numerical Examples section, it is evident that SIFT performance is not better than {\em fine tuned} windows for SST. However, when the window is not well chosen for SST, the reconstruction result is worse than that of SIFT. Nevertheless, SIFT does not depend on any window selection. This window-free property is the main benefit of SIFT. On the other hand, the SIFT-TFR has a higher quality compared with the TFR generated by SST. Indeed, when we apply the SST to a summation of various IMT functions, the spectral content of different IMT functions might ``talk to each other'' if the window is not selected properly. However, how to choose a suitable window is not an easy task. One approach is designing the window width locally based on the Reyni entropy \cite{sheu2017entropy}. However, it does not fully relieve this spectral leakage issue. In SIFT-TFR, the IMT functions are first decomposed iteratively starting from that with the highest frequency by ALIF, so the TFR of each estimated IMT function is less impacted by this spectral leakage issue.

Regarding the optimal boundary handling which minimize the errors induced in the decompositions, this is, to the best of our knowledge, an open problem in the field. Recently a few papers on this topic have been published in the literature \cite{cicone2020study,stallone2020new}. However, more has to be done in this direction. We plan to study this problem in a future work.

Another open problem that needs to be addressed is the identification of the presence of oscillatory components inside a signal under study, in order to decide if applying a decomposition technique like SIFT or not. How to identify an oscillatory component satisfying the IMT function model for the statistical inference purpose is very limited discussed in the literature, except \cite{sourisseau2019}. Even if we can state with confidence that there is an IMT function in the signal, it is common in practice that the oscillatory component may not always exist. How to detect when an oscillatory component disappears or appears is usually understood as the change point detection problem. To the best of our knowledge, this kind of change point detection problem for the oscillatory component was never studied in the past, except some recent efforts \cite{Tommy2020}. We will explore more in this statistical inference direction in our future work.

\section*{Acknowledgements}

Antonio Cicone is a member of the Italian ``Gruppo Nazionale di Calcolo Scientifico'' (GNCS).

\bibliography{Biblio_v3}
\bibliographystyle{plain}

\end{document}